\documentclass{amsart}
\usepackage{amssymb}
\usepackage{amsmath}
\usepackage{amscd}
\usepackage{url}
\usepackage{bbm}

\newtheorem{prop}{Proposition}[section]
\newtheorem{prop*}{Proposition*}[section]

\newtheorem{conj}[prop]{Conjecture}

\newtheorem{lem}[prop]{Lemma}
\newtheorem{lem*}[prop]{Lemma*}
\newtheorem{thm}[prop]{Theorem}
\newtheorem{thm*}[prop]{Theorem*}

\theoremstyle{definition}

\newtheorem{remar}[prop]{Remark}
\newcommand{\Ker}{\mathrm {Ker}}

\newcommand{\Aut}{{\mathrm {Aut}}}

\newcommand{\ord}{{\mathrm {ord}}}

\newcommand{\Hom}{{\mathrm {Hom}}}

\newcommand{\tr}{{\mathrm {tr}}}
\newcommand{\Tr}{{\mathrm {Tr}}}

\newcommand{\crys}{{\mathrm {crys}}}

\newcommand{\spin}{{\mathrm {spin}}}

\newcommand{\dR}{{\mathrm {dR}}}

\newcommand{\Frob}{{\mathrm {Frob}}}

\newcommand{\Gal}{\mathrm {Gal}}

\newcommand{\A}{{\mathbb A}}
\newcommand{\CC}{{\mathbb C}}
\newcommand{\C}{{\mathbb C}}
\newcommand{\RR}{{\mathbb R}}

\newcommand{\PP}{{\mathbb P}}
\newcommand{\QQ}{{\mathbb Q}}
\newcommand{\Q}{{\mathbb Q}}
\newcommand{\ZZ}{{\mathbb Z}}
\newcommand{\Z}{{\mathbb Z}}

\newcommand{\OOO}{{\mathcal O}}

\newcommand{\FF}{{\mathbb F}}

\newcommand{\GL}{\mathrm {GL}}

\newcommand{\SL}{\mathrm {SL}}
\newcommand{\Sp}{\mathrm {Sp}}

\newcommand{\GSp}{\mathrm {GSp}}

\newcommand{\Qbar}{\overline{\mathbb Q}}

\newcommand{\Fbar}{\overline{\mathbb F}}
\newcommand{\rhobar}{\overline{\rho}}

\newcommand{\HV}{\mathrm{HV}}
\newcommand{\divv}{\mathrm{div}}
\newcommand{\aaaa}{\mathbf{a}}

\DeclareFontFamily{U}{wncy}{}
    \DeclareFontShape{U}{wncy}{m}{n}{<->wncyr10}{}
    \DeclareSymbolFont{mcy}{U}{wncy}{m}{n}
    \DeclareMathSymbol{\Sh}{\mathord}{mcy}{"58} 

\begin{document}
\title{Modularity of a certain ``rank-$2$ attractor'' Calabi-Yau threefold}
\author{Neil Dummigan}

\date{September 23rd, 2026.}
\address{University of Sheffield\\ School of Mathematical and Physical Sciences\\
Hicks Building\\ Hounsfield Road\\ Sheffield, S3 7RH\\
U.K.}
\email{n.p.dummigan@shef.ac.uk}

\begin{abstract} We attempt to prove that the $4$-dimensional Galois representations associated with a certain Calabi-Yau threefold are reducible, with $2$-dimensional composition factors coming from specific modular forms of weights $2$ and $4$, both level $14$. This was essentially conjectured by Meyer and Verrill. It was revisited in its present form by Candelas, de la Ossa, Elmi and van Straten, whose computations of Euler factors in a whole pencil of Calabi-Yau threefolds highlighted this fibre as one of three overwhelmingly likely to be ``rank-$2$ attractors''.  {\bf Unfortunately there is an error in Section 7.}
\end{abstract}

\subjclass[2020]{11G40, 11F80}

\keywords{Modularity, Calabi-Yau, rank-$2$ attractor}

\maketitle

\section{Introduction} Let $X/\CC$ be a Calabi-Yau threefold, i.e. a smooth, proper, three-dimensional algebraic variety with trivial canonical bundle (implying the Hodge number $h^{3,0}=1$), and further $h^{0,1}=h^{0,2}=0$.
Note that $h^{p,q}=h^{q,p}$. The singular cohomology $H^3_B(X(\CC),\Q)$ has a Hodge decomposition
$$H^3_B(X(\CC),\Q)\otimes\CC= \oplus_{i=0}^3 H^{3-i,i}.$$
In special cases it can happen that there is a decomposition {\em over $\Q$}, 
\begin{equation}\label{rat} H^3_B(X(\CC),\Q)=V\oplus W, \,\,\,\,\text{with }V\otimes\CC=H^{3,0}\oplus H^{0,3}\,\,\text{and } W\otimes\CC=H^{2,1}\oplus H^{1,2}.\end{equation}
If $X$ is rigid (i.e. $h^{1,2}=h^{2,1}=0$) then of course there is vacuously such a decomposition. For rigid $X/\Q$, Gouv\^ea and Yui \cite[Theorem 3]{GY} (using an idea of Serre) and Dieulefait \cite{Di} (using an alternative argument) have proved that the representation of the absolute Galois group $G_{\Q}$ on the $2$-dimensional $\ell$-adic vector space $H^3_{\text{\'et}}(X_{\Qbar},\Q_{\ell})$ (for any prime number $\ell$) is isomorphic to that associated with some weight-$4$ cuspidal newform for $\Gamma_0(N)$, where $N$ is divisible at most by primes of bad reduction. 

Hulek and Verrill \cite{HV} constructed a $5$-dimensional family of proper $3$-dimensional varieties $\overline{X}_{\mathbf{a}}$, obtained from hypersurfaces $X_{\mathbf{a}}$ in a certain $4$-dimensional toric variety, by blowing them up along $10$ surfaces containing $30$ nodes lying outside the torus $T: X_1X_2X_3X_4X_5\neq 0$. The intersection with the torus is given by
$$(X_1+X_2+X_3+X_4+X_5)\left(\frac{a_1}{X_1}+\frac{a_2}{X_2}+\frac{a_3}{X_3}+\frac{a_4}{X_4}+\frac{a_5}{X_5}\right)=a_6.$$ For generic choices of $\mathbf{a}$, $\overline{X}_{\mathbf{a}}$ is a (smooth) Calabi-Yau threefold, with $h^{1,1}=45, h^{1,2}=5$. 

For various singular subfamilies, Hulek and Verrill constructed a Calabi-Yau resolution $\widetilde{\overline{X}}_{\mathbf{a}}$. They showed that this is rigid for four values $\mathbf{a}=
(1:1:1:1:1:1)$, $(1:1:1:1:1:9)$, $(1:1:1:1:4:4)$ and $(1:1:1:4:4:9)$ \cite[Corollary 4.9]{HV}. They also found three non-rigid examples for which (the semi-simplification of) the $\ell$-adic $G_{\Q}$-representation  $H^3_{\text{\'et}}(\widetilde{\overline{X}}_{\mathbf{a},\Qbar},\Q_{\ell})$ (conjecturally semi-simple anyway) splits as a direct sum $V_{\mathbf{a}}\oplus W_{\mathbf{a}}$, where $V_{\mathbf{a}}$ is $2$-dimensional and $W_{\mathbf{a}}$ is $2h^{1,2}$-dimensional. These are $(\mathbf{a}, h^{1,2})=((1:1:1:1:1:25), 4)$, $((1:1:1:9:9:9), 2)$ and $((1:1:4:4:4:16), 1)$ \cite[Corollary 5.12]{HV}. Their proof is geometric, using ruled elliptic surfaces inside $X_{\mathbf{a}}$, and it follows from this method that $W_{\mathbf{a}}\simeq h^{1,2}\rho_{g,\ell}(-1)$ (a Tate twist), where $\rho_{g,\ell}$ is the $2$-dimensional $\ell$-adic Galois representation attached to a weight-$2$ cuspidal newform $g$ for $\Gamma_0(30)$, associated with a certain isogeny class of elliptic curves over $\Q$. (See the discussion before \cite[Theorem 6.11]{HV}.) They also showed, using the Faltings-Livn\'e-Serre method, that $V_{\mathbf{a}}\simeq \rho_{f,\ell}$, where $f$ is a weight-$4$ cuspidal newform for $\Gamma_0(N)$, with $N=30, 90, 30$, respectively \cite[Theorem 6.11]{HV}. It follows from the geometric nature of their construction that there is an associated rational decomposition $H^3_B(\widetilde{\overline{X}}_{\mathbf{a}}(\CC),\Q)=V\oplus W$ as in (\ref{rat}).

Candelas, de la Ossa, Elmi and van Straten, in a recent paper \cite{COES}, looked at the subfamily $\mathbf{a}=(1:1:1:1:1:\phi^{-1})$ (in their notation). They consider the quotient of $\overline{X}_{\mathbf{a}}$ by a cyclic group $G$ of order $10$ of automorphisms, generated by a cyclic permutation of the variables and by the involution $X_i\mapsto 1/X_i$. They call this quotient $X_{\phi}$, and from now on we shall adopt this notation. For $\phi\notin\Sigma:=\{0,1,\infty, \frac{1}{9}, \frac{1}{25}\}$, $X_{\phi}$ is nonsingular, and is a Calabi-Yau threefold with $h^{1,2}=h^{2,1}=1$. 

Using a $p$-adic method described in \cite{COS}, which employs the Picard-Fuchs differential equation for the periods of the one-parameter family $X_{\phi}$, they determined the characteristic polynomial of Frobenius on $H^3_{\text{\'et}}(X_{\phi,\Qbar}, \Q_{\ell})$, for the first thousand primes $p\geq 5$. For each such $p$, they computed characteristic polynomials for all possible $\phi\pmod{p}$. By observing ``persistent factorisations'' of these polynomials, they produced overwhelming numerical evidence for the following.
\begin{conj}\label{Conj1} The (semi-simplification of the) $4$-dimensional representation of $G_{\Q}$ on $H^3_{\text{\'et}}(X_{-\frac{1}{7},\Qbar}, \Q_{\ell})$ is isomorphic to a direct sum $\rho_{f,\ell}\oplus \rho_{g,\ell}(-1)$, where $f$ and $g$ are cuspidal newforms of weights $4$ and $2$, respectively, for $\Gamma_0(14)$, labelled $\mathbf{14.4.a.a}$ and $\mathbf{14.2.a.a}$ in the LMFDB database \cite{LMF}.
\end{conj}
Note that both these forms have rational Hecke eigenvalues.
\begin{conj}\label{Conj2} Let $\phi_{\pm}:=33\pm 8\sqrt{17}$. The (semi-simplification of the) $4$-dimensional representation of $G_{\Q(\sqrt{17})}$ on $H^3_{\text{\'et}}(X_{\phi_{\pm},\Qbar}, \Q_{\ell})$ is isomorphic to the restriction from $G_{\Q}$ to $G_{\Q(\sqrt{17})}$ of $\rho_{f,\ell}\oplus \rho_{g,\ell}(-1)$, where $f$ and $g$ are cuspidal newforms of weights $4$ and $2$, respectively, for $\Gamma_1(34)$, character $\left(\frac{17}{\cdot}\right)$, labelled $\mathbf{34.4.b.a}$ and $\mathbf{34.2.b.a}$ in the LMFDB database \cite{LMF}.
\end{conj}
The fields generated by the Hecke eigenvalues of $\mathbf{34.4.b.a}$ and $\mathbf{34.2.b.a}$ are $\Q(i)$ and $\Q(\sqrt{-2})$, respectively. Note that (as observed in \cite[\S 6]{COES}) Conjecture \ref{Conj1} is essentially (before quotienting out by the automorphism group of order $10$) already contained in Meyer's thesis \cite[p.157]{Me}. See the first line of the third table in \cite[\S 5.8]{Me}, which is the result of collaboration between Meyer and Verrill.

\begin{remar} The table in \cite[\S 4.4]{Me} identifies two Calabi-Yau threefolds that appear to be rigid, and modular with respect to the form $\mathbf{14.4.a.a}$ (called 14/1 there) appearing in Conjecture \ref{Conj1}. If the rigidity were to be proved, then the connection with the form would easily be verified using the Faltings-Livn\'e-Serre method \cite[\S 1.5.1]{Me}.\hfill$\triangle$
\end{remar}
 
The authors of \cite{COES} also have compelling numerical evidence for a splitting $H^3_B(X_{\phi}(\CC), \Q)=V\oplus W$ as in (\ref{rat}), for these three values of $\phi$, and that they are ``rank-$2$ attractor points''. For a concise introduction to the attractor mechanism, see \cite[\S 1.2]{COES}. According to their preamble, it was first described by Ferrara, Kallosh and Strominger \cite{FKS} in the context of $N=2$ supergravity, and it ``remains a fascinating topic linking 4D black holes to string theory and has led to an understanding of black hole entropy in terms of the counting of microstates''. See the references in \cite[\S 1.1]{COES} for more on that. For the relationship between {\em rank-$2$ } attractors in particular and supersymmetric flux vacua, see \cite[Appendix C]{BEKK}.

The authors of \cite{COES} wished to find rank-$2$ attractor points in the family $X_{\phi}$. At such a point there is necessarily a splitting $H^3_B(X_{\phi}(\CC), \Q)=V\oplus W$ as in (\ref{rat}). The Hodge conjecture implies that this splitting should be motivic in nature, therefore giving an associated decomposition of Galois representations. So to find the values of $\phi$, they used characteristic polynomials of Frobenius, as indicated above. Though of course any $\phi$ revealed by this method will be an algebraic number, a conjecture of Moore \cite[Attractor Conjecture 8.2.1]{Moo} asserts that this is necessarily true of any attractor point.

In the examples examined by Hulek and Verrill, as already mentioned, the decomposition of the Galois representation was proved rigorously, by an algebro-geometric construction. So far, nobody has been able to prove Conjectures \ref{Conj1} and \ref{Conj2} by similar methods. (See \cite[\S 6]{COES} for a discussion.) The construction by Hulek and Verrill of ruled elliptic surfaces inside $\overline{X}_{\mathbf{a}}$ contributes, in the case $\mathbf{a}=(1,1,1,1,1,-7)$, a factor of $L(g, s-1)^4$ to the $L$-function of $\overline{X}_{\mathbf{a}}$. We shall see that the whole $L$-function is $L(f, s)L(g, s-1)^5$, but that it is the complementary factor $L(f,s)L(g,s-1)$ that produces the $L$-function of the quotient $X_{-1/7}$. In this paper, we attempt to close the gap between what the geometric construction provides and what we need to obtain a proof of Conjecture \ref{Conj1}, by taking a completely different approach, which we now outline.

The decomposition of Galois representations
$$H^3_{\text{\'et}}(X_{-\frac{1}{7},\Qbar}, \Q_{\ell})^{\mathrm{ss}}\simeq \rho_{f,\ell}\oplus \rho_{g,\ell}(-1)$$ is equivalent to
$$\tr(\Frob_p^{-1}\mid H^3_{\text{\'et}}(X_{-\frac{1}{7},\Qbar}, \Q_{\ell})=a_p+pb_p,$$ for all but finitely many $p$, where $f=\sum a_nq^n$ and $g=\sum b_nq^n$ (normalised newforms). Both sides are independent of $\ell$, the left hand side at least for primes $p\neq\ell$ of good reduction. So it suffices to prove the decomposition for $\ell=5$.  

Let $\rho_{\ell}$ be the representation of $G_{\Q}$ on $H^3_{\text{\'et}}(X_{-\frac{1}{7},\Qbar}, \Z_{\ell})$, with $\rhobar_{\ell}$ the residual representation. We show that $\rhobar_{5}$ is reducible, with composition factors of dimensions $1$, $1$ and $2$. In fact they can be identified with characters $\mathrm{id}$, $\epsilon^{-3}$ (where $\epsilon$ is the mod $5$ cyclotomic character) and $\rhobar_{g, 5}(-1)$. 

To show that $\rhobar_5$ has a $G_{\Q}$-invariant filtration with subquotients of dimensions $1$, $1$, $2$, we look at the monodromy action of $\pi_1(\PP^1(\CC)-\Sigma)$ on the space of $\rhobar_5$, identifying its image. Using the fact that the Galois action must normalise this image, we can place the image of $G_{\Q}$ inside an appropriate maximal parabolic subgroup of $\GSp_4(\FF_5)$. This is something special to $\ell=5$, and does not depend on the choice $\phi=-\frac{1}{7}$. This happens in Section 3. At this point I should acknowledge the good influence of Vasily Golyshev, who has impressed upon me the importance of this kind of residual shrinking of monodromy image, and consequent reducibility of Galois representations across a parametrised family. 

We must be careful to show that the $\Z_5$-lattice, spanned by the basis with respect to which the monodromy matrices in \cite{COES} are computed, really is $H^3_{\text{\'et}}(X_{-\frac{1}{7},\Qbar}, \Z_{5})$, hence is $G_{\Q}$-invariant as well as monodromy invariant. This is dealt with in Section 2.

In Section 5 (for $\phi=-\frac{1}{7}$) we identify the composition factors of $\rhobar_5$, using the fact that their conductors have prime factors at most $2$ and $7$ (the primes of bad reduction), whose exponents can be bounded, and various computations to eliminate all alternative possibilities. Once we have proved that the $2$-dimensional subfactor is irreducible, its modularity (a consequence of what used to be Serre's conjecture) is an important fact.

If we knew that the $5$-adic representation $\rho_5$ were reducible, we would be able, in Section 8, to prove that it has $2$-dimensional composition factors. One of them has irreducible residual representation $\rhobar_{g,5}(-1)$. The other has reducible residual representation, factors $\mathrm{id}$ and $\epsilon^{-3}$. Modularity (up to twist) of both $2$-dimensional $5$-adic representations would follow from work of Kisin, Skinner-Wiles and Pan on the Fontaine-Mazur conjecture for $\GL_2$. Using results of Livn\'e to bound the levels of the modular forms in question (i.e. the conductors of the $5$-adic representations), they are easily proved to be those appearing in Conjecture \ref{Conj1}.

It remains to prove that $\rho_5$ is reducible. To do this, we assume that it is irreducible, seeking a contradiction. Then, following Brown \cite{Br} in adapting a construction of Ribet \cite{Ri}, we would like to use a non-split extension between the factors $\mathrm{id}$ and $\rhobar_{g,5}(-1)$ of $\rhobar_5$ to construct an element of order $5$ in a certain Bloch-Kato Selmer group. I wanted to do this in Section 7. Work of Kato, using his Euler system \cite{Ka}, bounds orders of Selmer groups using $L$-values. We might hope to get a contradiction by showing that an appropriate algebraic part of the corresponding $L$-value is not divisible by $5$. 

A difficulty is that the $L$-value $L(g, -1)$, related to the Selmer group in question by the Bloch-Kato conjecture \cite{BK}, is not critical, making the algebraic factor difficult to extract, and Kato's theorem not directly applicable. Inspired by work of Berger and Klosin \cite[\S 2.3]{BeKl}, we get around this by using $p$-adic $L$-functions and Iwasawa theory. Essentially, the $(-2)$-Tate twist that takes us from the central point $s=1$ to $L(g,-1)$ is replaced by a congruent (mod $5$) twist by the Legendre symbol of conductor $5$, allowing us to look at the central critical value of a twisted $L$-function instead. A complication arising in our case is that the elliptic curve associated with $g$ has supersingular reduction at $5$. This forces us to employ the $\pm$-theory developed by Pollack, Kobayashi and others, on $p$-adic $L$-functions and Selmer groups in the supersingular case. What we need is introduced in Section 6.

The strong experimental evidence for Conjectures \ref{Conj1} and \ref{Conj2} gives good cause for them to be widely believed, by mathematicians as well as by physicists. The authors of \cite{COES} have surely found the rank-$2$ attractor points they were looking for. My goal in this paper was nonetheless to prove Conjecture \ref{Conj1}  and put their discovery on a firm mathematical foundation. (I hasten to add that we cannot deduce from Conjecture \ref{Conj1} the expected  decomposition (\ref{rat}) in singular cohomology.) 

At several places in Section 5, and to a lesser extent in Section 8, we need to know characteristic polynomials of Frobenius for certain $p$. The fact that the computations of these in \cite{COES} agree with Conjectures \ref{Conj1} and \ref{Conj2}, and with known values of Hecke eigenvalues, is ample reason to believe that they are correct. But the validity of the $p$-adic method used relies on an unproved conjecture, in \cite[\S 4.4]{COS}, in particular on the exact form of a certain $4$-by-$4$ matrix $U(0)$. Therefore we avoid relying on these computations. Instead, in Section 4 we determine the Euler factors in question by counting points on the variety $\overline{X}_{\mathbf{a}}$, and by following Hulek and Verrill in identifying a factor $L(g, s-1)^4$ in the $L$-function arising from $H^3_{\text{\'et}}(\overline{X}_{\mathbf{a},\Qbar}, \Q_{\ell})$. 

{\bf There is an error in Section 7, which will be explained there. } The results of Sections 2--6, and the conditional result of Section 8, are, as far as I know, correct, so I leave them in case they (and various remarks) are of independent interest.

\section{An explicit basis for $5$-integral homology} 

We are looking at the degree $3$ integral homology of the fibre $X_{\phi}(\CC)$. Throughout, we mod out by any torsion, without reflecting this in the notation. In other words, we identify the integral homology with its image in the rational homology. We shall introduce some explicit cycles representing classes in $H_3(X_{\phi}(\CC), \ZZ)$. 

Recall that the equation for $\overline{X}_{\mathbf{a}}\cap T$ is 
$$(X_1+X_2+X_3+X_4+X_5)\left(\frac{1}{X_1}+\frac{1}{X_2}+\frac{1}{X_3}+\frac{1}{X_4}+\frac{1}{X_5}\right)=\phi^{-1}.$$
Dehomogenising with respect to $X_5$, this becomes 
\begin{equation}\label{XT}(x_1+x_2+x_3+x_4+1)\left(\frac{1}{x_1}+\frac{1}{x_2}+\frac{1}{x_3}+\frac{1}{x_4}+1\right)=\phi^{-1}.\end{equation}
If we choose $\phi$ with $|\phi|<1/25$, so that $|\phi^{-1}|>25$, then clearly there is no solution with $|x_1|=|x_2|=|x_3|=|x_4|=1$, since both $(x_1+x_2+x_3+x_4+1)$ and $\left(\frac{1}{x_1}+\frac{1}{x_2}+\frac{1}{x_3}+\frac{1}{x_4}+1\right)$ would be at most $5$ in absolute value. 

Letting $a:=1+x_1+x_2+x_3, b:=1+\frac{1}{x_1}+\frac{1}{x_2}+\frac{1}{x_3},$ we get a quadratic equation \cite[(33)]{HV}
\begin{equation}\label{XTquad} bx_4^2+(ab+1-\phi^{-1})x_4+a=0.\end{equation}
The product of the roots is $a/b$. In the case that $|x_1|=|x_2|=|x_3|=1$ we have $b=\overline{a}$, so $|a/b|=1$.
Note that since $\phi\neq 1$, $a=b=0$ is impossible. There must be a unique root $x_4=:s(x_1, x_2, x_3)$ with $|x_4|<1$. (Recall that $|x_4|\neq 1$.) Then we define a $3$-cycle (and the class in $H_3(\overline{X}_{\mathbf{a}}(\CC), \ZZ)$ it represents) by
$$\tau':=\{(x_1,x_2,x_3,x_4,1)\in (\overline{X}_{\mathbf{a}}\cap T')(\CC):\,|x_1|=|x_2|=|x_3|=1, x_4=s(x_1,x_2,x_3)\}.$$
Topologically, this is a $3$-dimensional torus. 

We have constructed a cycle in the Hulek-Verrill manifold $\overline{X}_{\mathbf{a}}(\CC)$, but the condition $|x_1|=|x_2|=|x_3|=|x_5|=1, |x_4|<1$ implies that any orbit of the group $G$ (generated by inversion and a cyclic shift) has at most one point in $\tau'$, so $\tau'$ projects bijectively to its image in $X_{\phi}(\CC)$, which we call $\tau$.

Consider now the positive real locus 
$$\sigma'':=\{(x_1, x_2, x_3, x_4, 1)\in (\overline{X}_{\mathbf{a}}\cap T)(\CC):\, x_i\in \RR_{>0},\,\forall \,1\leq i\leq 4\}.$$
We take $\phi$ to be real and just slightly smaller than $1/25$, so that $\phi^{-1}-25$ is a small positive real number. If $\phi$ were equal to $1/25$ then the point $(1,1,1,1,1)$ would be a nodal singularity on $\overline{X}_{\mathbf{a}}$, so (following \cite[(3.2)]{COGP}) we set $x_i=1+y_i$ and expand the equation (\ref{XT}) to second order in the $y_i$, obtaining an approximation
$$\left(5+\sum y_i\right)\left(5-\sum y_i+\sum y_i^2\right)\approx\phi^{-1},$$ i.e.
\begin{equation}\label{sphere} 4(y_1^2+y_2^2+y_3^2+y_4^2)-2(y_1y_2+y_1y_3+y_1y_4+y_2y_3+y_2y_4+y_3y_4)\approx \phi^{-1}-25.\end{equation}
Summing $y_i^2+y_j^2-2y_iy_j\geq 0$ over all $6$ unordered pairs $\{i,j\}$ of distinct $1\leq i,j\leq 4$, we find that
$$3(y_1^2+y_2^2+y_3^2+y_4^2)-2(y_1y_2+y_1y_3+y_1y_4+y_2y_3+y_2y_4+y_3y_4)\geq 0,$$
so the quadratic form on the left of (\ref{sphere}) is positive definite, and it describes a topological $3$-sphere.
If all $y_i\ll 1$ then this is a good approximation. If we make $\phi^{-1}-25$ small enough that, subject to $y_i\ll 1$, they are forced to be even smaller, then there will be an annulus from which $(y_1,y_2,y_3,y_4)$ is excluded, and we can be sure that, whether or not we are looking at the whole of $\sigma''$, there is at least a connected component of $\sigma''$ approximately described by (\ref{sphere}). We call this $\sigma'$, and by triangulating this orientable $3$-manifold we create a $3$-cycle representing a class in $H_3(\overline{X}_{\mathbf{a}}(\CC), \ZZ)$, which we give the same name. Let $\sigma$ be the image of the manifold $\sigma'$ under projection from $\overline{X}_{\mathbf{a}}(\CC)$ to $X_{\phi}(\CC)$. 

As a $3$-sphere, $\sigma'$ is orientable. We need to show that the automorphisms in $G$ are orientation-preserving to deduce that $\sigma$ is also orientable. A $5$-cycle has odd order, so is automatically orientation-preserving. The inversion sending each $x_i$ to $1/x_i$ is approximated by $y_i\mapsto -y_i$ for small $y_i$, hence by the antipodal map of the standard $3$-sphere, but such antipodal maps are orientation-preserving in odd dimension. Hence $\sigma$ is orientable, and we may create a corresponding element of $H_3(X_{\phi}(\CC), \ZZ)$ with the same name. Note that, since the projection map is $10$-to-$1$ from the manifold $\sigma'$ onto its image $\sigma$, the push-forward of the homology class $\sigma'$ is $10\sigma$.

\begin{lem}\label{st} Choose $\phi$ real and just slightly smaller than $1/25$, and $\tau, \sigma$ as above. Then for the intersection pairing $\langle,\rangle$ on $H_3(X_{\phi}(\CC),\ZZ)$ we have $\langle\sigma,\tau\rangle=\pm 1$.
\end{lem} (See Remark \ref{sgn} below for the determination of the exact sign.)
\begin{proof} The unique intersection point of $\tau'$ and $\sigma'$ is $(1,1,1,s(1,1,1),1)$. The lemma now follows from the fact that the manifold $\tau'$ projects bijectively to its image $\tau$. 
\end{proof}
These homology classes, and their intersection relation, may be parallel-transported to other choices of $\phi\notin\Sigma$. Clearly $\sigma'$ is a vanishing cycle as $\phi\rightarrow\frac{1}{25}$.

 Let $f_0(\phi):=\sum_{n\geq 0} a_n\phi^n$, with $a_n=\sum_{i+j+k+l+m=n}\left(\frac{n!}{i!j!k!l!m!}\right)^2$.
\begin{lem}\label{Holper} There exists a holomorphic $3$-form $\omega_{\phi}$ on $X_{\phi}$, depending algebraically on $\phi$, such that $\int_{\tau}\omega_{\phi}=f_0(\phi)$.
\end{lem}
\begin{proof} It suffices to prove that there exists a holomorphic $3$-form $\omega'_{\phi}$ on $\overline{X}_{\aaaa}$, depending algebraically on $\phi$, such that $\int_{\tau'}\omega'_{\phi}=f_0(\phi)$, since then its symmetrization under the action of $G$ will be the pullback of a form $\omega_{\phi}$ on $X_{\phi}$ with the required property.

Recall that $\overline{X}_{\aaaa}$ is obtained from $X_{\aaaa}$ by blowing it up along $10$ surfaces containing $30$ nodes lying outside the torus $T$. (See \cite[3.2.4, 3.2.5]{HV} for details.) This is a ``small'' projective resolution, the inverse image in $\overline{X}_{\aaaa}$ of each node in $X_{\aaaa}$ being a projective line. Let $\pi: \overline{X}_{\aaaa}\rightarrow X_{\aaaa}$ be the resolution morphism, $N\subseteq X_{\aaaa}$ the set of nodes, $\overline{N}:=\pi^{-1}(N)$, a union of projective lines. Then $\pi: \overline{U}\simeq U$, where $U:=X_{\aaaa}\backslash N$ and $\overline{U}:=\overline{X}_{\aaaa}\backslash\overline{N}$. The canonical bundle $\Omega^3_{\overline{X}_{\aaaa}}$ on the nonsingular variety $\overline{X}_{\aaaa}$ is a reflexive sheaf (isomorphic to its double dual), as in \cite[\S 8.0]{CLS} or \cite[\S 1]{Ha2}. Since $\overline{N}$ is a codimension-$2$ subvariety of $\overline{X}_{\aaaa}$, it follows from \cite[Proposition 1.6]{Ha2} that the restriction map from $H^0(\overline{X}_{\aaaa},\Omega^3)$ to $H^0(\overline{U}, \Omega^3)\simeq H^0(U,\Omega^3)$ is surjective, So it suffices to produce $\omega'_{\phi}$ (by abuse of notation) in $H^0(U,\Omega^3)$, with the required integral over $\tau'$ (i.e. over $\pi(\tau')\simeq \tau'$).

In the notation of \cite{HV}, one considers a certain $4$-dimensional reflexive polytope $\Delta$, the dual polytope $\Delta^*$, the fan $\Sigma$ spanned by its faces, and the associated toric variety $P$, which is a singular Gorenstein Fano variety. Taking the maximal projective triangulation $\widetilde{\Delta^*}$ constructed in \cite[\S 2.1]{HV}, and the associated simplicial fan $\tilde{\Sigma}$, we get a smooth toric variety $\tilde{P}$, a resolution of $P$. 

Recall that the equation for $X_{\mathbf{a}}\cap T$ is 
$$(X_1+X_2+X_3+X_4+X_5)\left(\frac{1}{X_1}+\frac{1}{X_2}+\frac{1}{X_3}+\frac{1}{X_4}+\frac{1}{X_5}\right)=\phi^{-1},$$
and that dehomogenising with respect to $X_5$, this becomes 
$$(x_1+x_2+x_3+x_4+1)\left(\frac{1}{x_1}+\frac{1}{x_2}+\frac{1}{x_3}+\frac{1}{x_4}+1\right)=\phi^{-1}.$$
By \cite[Proof of Corollary 8.2.8]{CLS}, the differential form $\Omega:=\frac{dx_1}{x_1}\wedge\frac{dx_2}{x_2}\wedge\frac{dx_3}{x_3}\wedge\frac{dx_4}{x_4}$ has divisor $-\sum D_{\rho}$ on $\tilde{P}$, where $D_{\rho}$ is the $T$-invariant divisor obtained by setting the homogeneous coordinate $x_{\rho}=0$. Here $\rho$ runs over the $30$ $1$-dimensional cones of $\tilde{\Sigma}$ (or equally of $\Sigma$), which are in natural bijection with the $3$-dimensional faces of the polytope $\Delta$. The inward pointing normals are $u_{\rho}$, which are of the form $\pm e_i$ and $\pm(e_i+e_j)$, with $1\leq i\neq j\leq 5$ and $e_i$ the unit vector with $1$ in the $i^{\mathrm{th}}$ position, $0$ elsewhere. (Strictly speaking, we should project these vectors to the hyperplane $\{\sum a_ie_i:\,\,\sum a_i=0\}$, inside which these $4$-dimensional polytopes and fans lie.)

We would like to show that the function $$f=(X_1+X_2+X_3+X_4+X_5)\left(\frac{1}{X_1}+\frac{1}{X_2}+\frac{1}{X_3}+\frac{1}{X_4}+\frac{1}{X_5}\right)-\phi^{-1}$$
$$=(x_1+x_2+x_3+x_4+1)\left(\frac{1}{x_1}+\frac{1}{x_2}+\frac{1}{x_3}+\frac{1}{x_4}+1\right)-\phi^{-1}$$
also (like $\Omega$) has order exactly $-1$ at each $D_{\rho}$. By symmetry it suffices to show this for the four choices  $u_{\rho}=-(e_1+e_2), e_4+e_5, -e_1, e_5$. We work in the $4$-dimensional affine patch associated with one of the $120$ $3$-dimensional cones of $\tilde{\Sigma}$, for which coordinates are $x=X_1/X_2, y=X_2/X_3, z=X_3/X_4$ and $w=X_4/X_5$, as in \cite[\S 2.2]{HV}. In fact, each of these four divisors $D_{\rho}$ is defined by the vanishing of one of these coordinates. For example, when $u_{\rho}=-(e_1+e_2)$, the orders at $D_{\rho}$ of these coordinate functions are obtained by evaluating $-(e_1+e_2)$ on them, using the rule $e_i(X_j^a)=a\delta_{ij}$. Thus $\ord_{D_\rho}(y)=-1$, while $\ord_{D_\rho}(x)=\ord_{D_\rho}(z)=\ord_{D_\rho}(w)=0$. As in \cite[(11)]{HV},
$$f=\frac{(1+x+xy+xyz+xyzw)(1+w+wz+wzy+wzyx)}{xyzw}-\phi^{-1},$$ so clearly $\ord_{D_{\rho}}(f)=-1$ for each of these $D_{\rho}$, hence by symmetry, using different affine patches, for all $D_{\rho}$.

In fact, $\divv(f)=X_{\aaaa}-\sum D_{\rho}$. It follows that the $4$-form $\Omega/f$ is regular on the open set $\tilde{P}\backslash X_{\aaaa}$. We can now define the required $\omega'_{\phi}$ on $U=X_{\aaaa}\backslash N$ as essentially the Poincar\'e residue  of $\frac{-\Omega}{\phi (2\pi i)^3 f}$ \cite[\S 8]{Gr}. At a point of $U\cap T$, without loss of generality $\frac{\partial f}{\partial x_4}\neq 0$, so $$\Omega/f=\frac{1}{\partial f/\partial x_4}\frac{dx_1}{x_1}\wedge\frac{dx_2}{x_2}\wedge\frac{dx_3}{x_3}\wedge\frac{df}{f}.$$ Using $f$ as a local coordinate taking the value $0$ on $U$, integrating along a circle centred at $0$ in the complex $f$-plane (and dividing by $2\pi i$) is equivalent to removing the factor $\frac{df}{f}$. Thus (where $\frac{\partial f}{\partial x_4}\neq 0$) $$\omega'_{\phi}=\frac{-1}{\phi(2\pi i)^3\partial f/\partial x_4}\frac{dx_1}{x_1}\wedge\frac{dx_2}{x_2}\wedge\frac{dx_3}{x_3},$$ which certainly depends algebraically on $\phi$. In neighbourhoods of points of $U\backslash T$ we may argue similarly, using local coordinates different from $x_1,\ldots x_4$, and local equations for $X_{\aaaa}$ different from $f=0$. Of course all these local constructions glue together compatibly.

The integral of $\omega'_{\phi}$ over $\tau'$ is the same as the integral of $\frac{-1}{\phi (2\pi i)^4}\Omega/f$ over a tube about $\tau'$. Recalling that $x_4=s(x_1,x_2,x_3)$ satisfies $|x_4|<1$, we may take this tube to be $|x_1|=|x_2|=|x_3|=|x_4|=1$.
Then we may follow the computation of Verrill \cite[Propositions 5, 6]{Ve}, but with an additional variable. Thus $\int_{\tau'}\omega'_{\phi}$ becomes
$$\frac{1}{(2\pi i)^4}\iiiint_{|x_1|=|x_2|=|x_3|=|x_4|=1}\frac{\Omega}{1-\phi(x_1+x_2+x_3+x_4+1)\left(\frac{1}{x_1}+\frac{1}{x_2}+\frac{1}{x_3}+\frac{1}{x_4}+1\right)}.$$
Expanding as a geometric series and integrating term-by-term, picking out the constant term in $\left[(x_1+x_2+x_3+x_4+1)\left(\frac{1}{x_1}+\frac{1}{x_2}+\frac{1}{x_3}+\frac{1}{x_4}+1\right)\right]^n$, we obtain the desired formula.
\end{proof}

\begin{lem}\label{diffeq} The period function $f_0(\phi)$ satisfies $\mathcal{L}(f_0)=0$, where $$\mathcal{L}=S_4\theta^4+S_3\theta^3+S_2\theta^2+S_1\theta+S_0,$$
with $\theta:=\phi\,\frac{d}{d\phi}$ and $S_0,\ldots, S_4$ certain cubic polynomials in $\phi$, with $S_4=(\phi-1)(9\phi-1)(25\phi-1)$ and the rest as in \cite[\S 3.1]{COES}. (This is number 34 in the AESZ database.)
\end{lem}
\begin{proof} The operator $\mathcal{L}$ is presented in \cite[\S 3.1]{COES}, with reference to the method described in \cite[\S 6.1]{COS}. Thus one uses sufficiently many coefficients of $f_0$ to solve for unknown $S_0,\ldots, S_4$ in an operator $\mathcal{L}$ that might annihilate it, then uses the closed form of the coefficients of $f_0$ to show that they satisfy the necessary recurrence relation for it to be killed by this $\mathcal{L}$.
\end{proof}

For a $\CC$-basis for $\ker(\mathcal{L})$ we may take the Frobenius basis centred at the maximum unipotent monodromy (MUM) point $\phi=0$:

\begin{align*}\Phi_0(\phi)&=f_0(\phi);\\
\Phi_1(\phi)&=f_0(\phi)\log(\phi)+f_1(\phi);\\
\Phi_2(\phi)&=\frac{1}{2}f_0(\phi)\log^2(\phi)+f_1(\phi)\log(\phi)+f_2(\phi);\\
\Phi_3(\phi)&=\frac{1}{6}f_0(\phi)\log^3(\phi)+\frac{1}{2}f_1(\phi)\log^2(\phi)+f_2(\phi)\log(\phi)+f_3(\phi).\end{align*}
Here the $f_i$ are certain power series in $\phi$, with $f_0(0)=1$ and $f_j(0)=0$ for $j>0$, and we fix a branch of $\log$.
We shall use a ``scaled Frobenius basis'' $\{\Psi_j\}$, with $\Psi_j:=\frac{\Phi_j}{(2\pi i)^j}$.

Viewing $\log$ now as a multi-valued function on $\CC\backslash\{0\}$, the $\Psi_j$ may be viewed as multi-valued functions on a sufficiently small punctured disc centred at $0$. Using $\log(\phi)\mapsto\log(\phi)+2\pi i$, an elementary calculation confirms that, with respect to the basis $\{\Psi_0,\Psi_1,\Psi_2,\Psi_3\}$, the monodromy for a generating loop $c_0$ around $0$ is represented by the matrix $$T_0=\begin{pmatrix}1&1&1/2&1/6\\0&1&1&1/2\\0&0&1&1\\0&0&0&1\end{pmatrix}=\exp(N),\text{  with  } N=\begin{pmatrix}0&1&0&0\\0&0&1&0\\0&0&0&1\\0&0&0&0\end{pmatrix}.$$

Monodromy matrices $T_{1/25}, T_{1/9}$, $T_1$ and $T_{\infty}$ about the other singular points may be produced similarly, with respect to local Frobenius bases. Then they can be expressed in the original basis $\{\Psi_j\}$ by conjugating by transition matrices, which are approximated by a process of numerical analytic continuation. This was first done by D. van Straten, ``almost 20 years ago'' at the time he wrote \cite{vS}.

The fundamental group $\pi_1(\PP^1(\CC)\backslash\Sigma)$ is generated by loops
$c_0, c_1, c_{\infty}, c_{1/9}, c_{1/25}$ about the singular parameter values, satisfying a relation $c_0c_1c_{\infty}c_{1/9}c_{1/25}=\mathrm{id}$. These act on $\ker(\mathcal{L})$ by the matrices $T_0$ etc., but they also act on the local system of homology groups $H_3(X_{\phi}(\CC), \ZZ)$.

\begin{lem}\label{T125}
$c_{1/25}(\tau)=\tau\mp 10\sigma.$
\end{lem}
\begin{proof} Up on $\overline{X}_{\mathbf{a}}$, $\sigma'$ is a vanishing cycle as $\phi\rightarrow \frac{1}{25}$, so by the Picard-Lefschetz formula we have $$c_{1/25}(\tau')=\tau'+\langle \tau',\sigma'\rangle \sigma'=\tau'\mp\sigma',$$ by the proof of Lemma \ref{st}. Pushing forward to $X_{\phi}$,
$$c_{1/25}(\tau)=\tau\mp 10\sigma,$$ as required.
\end{proof}

By abuse of notation, let $T_0$ denote the monodromy operator of $c_0$ on $H_3(X_{\phi}(\CC), \ZZ)$, and let $$N=\log(T_0)=\log(I+(T_0-I))=(T_0-I)-(T_0-I)^2/2+(T_0-I)^3/3,$$ noting that $(T_0-I)^4=0$ (using the earlier matrix). Then $T_0=\exp(N)$ and, although $N$ may not preserve $H_3(X_{\phi}(\CC), \ZZ)$, it is integral at least away from the primes $2$ and $3$, so in particular it acts on $H_3(X_{\phi}(\CC), \ZZ_{(5)})$, and on $H_3(X_{\phi}(\CC), \ZZ_5)$.

\begin{lem}\label{iota}
Given the algebraic family of regular $3$-forms $\omega_{\phi}$ on the $X_{\phi}$, as in Lemma \ref{Holper}, and a parallel section $\gamma_{\phi}$ of the locally constant sheaf $H_3(X_{\phi}(\CC), \ZZ)$, the map 
$$\gamma_{\phi}\mapsto \int_{\gamma_{\phi}}\omega_{\phi}$$ produces an isomorphism
$$\iota_{\omega}: H_3(X_{\phi}(\CC),\ZZ)\otimes\CC \simeq \ker(\mathcal{L}).$$
\end{lem}
 Here the right hand side of the isomorphism is viewed as a $4$-dimensional space of solutions $f$ to the equation $\mathcal{L}(f)=0$, defined for $\phi$ in some small open disc close to, but not containing, $\phi=0$. On the left hand side we may think of $\phi$ as fixed at any basepoint $\phi_0$ inside that small disc, or more generally any $\phi_0$ outside the singular locus $\Sigma:=\{0,1,\infty, \frac{1}{9}, \frac{1}{25}\}$, as long as we fix a homotopy class of paths from $\phi_0$ to some point in the disc, along which to parallel transport homology classes via the flat connection. 

\begin{proof} Part of what we must show is even that $\ker(\mathcal{L})$ is a codomain for $\iota_{\omega}$. In fact we shall prove that it is the image of $\iota_{\omega}$. After that, the fact that both sides are $4$-dimensional will then imply that $\iota_{\omega}$ must be an isomorphism between them. 

The map $\iota_{\omega}$ respects monodromy, i.e. $$\iota_{\omega}(c(\gamma))=c(\iota_{\omega}(\gamma)),\,\,\,\,\forall \gamma\in H_3(X_{\phi}(\CC),\ZZ),\,\,\,c\in \pi_1(\PP^1(\CC)\backslash\Sigma).$$ Since the monodromy action preserves $\ker(\mathcal{L})$, and since $\iota_{\omega}(\tau)=f_0\in\ker(\mathcal{L})$, letting $c=c_{1/25}$, Lemma \ref{T125} shows that also $\iota_{\omega}(\sigma)\in\ker(\mathcal{L})$. 

Using the approximation for $\iota_{\omega}(\sigma)$ in Proposition \ref{sigm} below, and the matrix $N$ preceding Lemma \ref{T125}, it is clear that $$\{\iota_{\omega}(\sigma), N(\iota_{\omega}(\sigma)), N^2(\iota_{\omega}(\sigma)), N^3(\iota_{\omega}(\sigma))\}$$ is linearly independent, so spans $\ker(\mathcal{L})$, but it is also the same as 
$$\{\iota_{\omega}(\sigma), (\iota_{\omega}(N\sigma)), (\iota_{\omega}(N^2\sigma)), (\iota_{\omega}(N^3\sigma))\}.$$This proves that the $4$-dimensional space $\ker(\mathcal{L})$ is contained in, hence equal to, the image of $\iota_{\omega}$.
\end{proof} 

This lemma justifies calling $\mathcal{L}(f)=0$ the Picard-Fuchs equation of the family $X_{\phi}$. It is satisfied by the periods of $\omega_{\phi}$ over all $3$-cycles in $H_3(X_{\phi}(\CC),\ZZ)$, not just by $f_0$ over $\tau$. It also shows that the action of monodromy on functions may be used to track the action on integral homology classes, if they (or rather their images under $\iota_{\omega}$) are expressed in terms of the scaled Frobenius basis $\{\Psi_j\}$. 

The following is \cite[\S 3.2]{COES}, justified by general expectations about mirror symmetry. Their $\Pi$ is our $\begin{pmatrix}v_0\\v_1'\\v_3\\v_2\end{pmatrix}$. Note also that in \cite[\S 3.2]{COES}, $\bar{\omega}_0=\Phi_0$ and $\bar{\omega}_1=\Phi_1$, but $\bar{\omega}_2=2\Phi_2$ and $\bar{\omega}_3=6\Phi_3$. So to get from their matrix $\hat{\rho}$ to ours, one must multiply the third column by $2$, the fourth column by $6$, and swap the third and fourth rows. (We also correct a sign error in the top left corner.)
\begin{conj}\label{basis} Let $$\begin{pmatrix}v_0\\v_1'\\v_2\\v_3\end{pmatrix}:=\begin{pmatrix}-8\lambda & 1/2 & 0 & 12\\1/2 & 0 & -12 & 0\\0 & 1 & 0 & 0\\1 & 0 & 0 & 0\end{pmatrix} \begin{pmatrix}\Psi_0\\\Psi_1\\\Psi_2\\\Psi_3\end{pmatrix},$$ where $\lambda:=\frac{\zeta(3)}{(2\pi i)^3}$. Then $\{v_0, v_1', v_2, v_3\}$ is the image, under the isomorphism $\iota_{\omega}$, of a $\ZZ$-basis of $H_3(X_{\phi}(\CC), \ZZ)$.
\end{conj}
 The conjecture is equivalent to the same property for $\{v_0, v_1, v_2, v_3\}$, where $v_1:=-v_1'+v_3$. Notice that $\{v_0, v_1, v_2, v_3\}=\{v_0, Nv_0, N^2v_0/12, \Phi_0\}$.

The following is a direct consequence of Lemmas \ref{Holper} and \ref{diffeq}.
\begin{lem}\label{holper} $\iota_{\omega}(\tau)=\Phi_0=v_3$.
\end{lem}

\begin{prop}\label{sigm} To a very good approximation,
$$\iota_{\omega}(\sigma)=\mp(8\lambda\Psi_0-\frac{1}{2}\Psi_1-12\Psi_3).$$ In particular (for use in the proof of Lemma \ref{12} below) the coefficient of $\Psi_3$ on the right hand side is accurate to within $1/12$.
\end{prop}
\begin{proof} The idea is due to Van Straten \cite{vS}, who computed, to a high degree of accuracy, that 
$$T_{1/25}(\Psi_0)\approx (80\lambda+1)\Psi_0-5\Psi_1-120\Psi_3.$$
Using Lemma \ref{holper} we can read this as $$c_{1/25}(\tau)\approx \tau+\delta,$$
where $$\iota_{\omega}(\delta)=80\lambda\Psi_0-5\Psi_1-120\Psi_3.$$
By Lemma \ref{T125},
$$c_{1/25}(\tau)=\tau\mp 10\sigma,$$ so
$\delta\approx\mp 10\sigma$.
\end{proof}
\begin{lem}\label{12}
$N^3\sigma=\pm 12\tau$.
\end{lem}
\begin{proof} From the earlier (preceding Lemma \ref{T125}) matrix for $N$ with respect to the basis $\{\Psi_0,\Psi_1,\Psi_2,\Psi_3\}$ of $\ker(\mathcal{L})\simeq H_3(X_{\phi}(\CC), \CC)$, we know that $N^4=0$ and that $\mathrm{Im}(N^3)=\ker(N)$ is $1$-dimensional, spanned by $\tau$ (equivalent to $\Psi_0$, according to Lemma \ref{holper}). So $N^3\sigma$ is necessarily a multiple of $\tau$. Using $N=(T_0-I)-(T_0-I)^2/2+(T_0-I)^3/3$, we see that $6N$ (like $T_0$) preserves $H_3(X_{\phi}(\CC), \ZZ)$, hence so too does $6^3N^3$. Using also $N^3=6(T_0-(I+N+N^2/2))$ we can refine this to $12N^3$ preserving $H_3(X_{\phi}(\CC), \ZZ)$. Hence if
$N^3\sigma=\mu\tau$ then $12\mu\in\ZZ$. (We know, by Lemma \ref{st}, that $\tau$ is not divisible in $H_3(X_{\phi}(\CC), \ZZ)$ by any integer other than $\pm 1$.)

Recall that, with respect to the basis $\{\Psi_0,\Psi_1,\Psi_2,\Psi_3\}$, $N$ is represented by the matrix $\begin{pmatrix}0&1&0&0\\0&0&1&0\\0&0&0&1\\0&0&0&0\end{pmatrix}.$ It follows from Proposition \ref{sigm} and Lemma \ref{holper} that $N^3\sigma\approx \pm 12\tau$, i.e. $\mu\approx \pm 12$, but the approximation is good enough that $12\mu\in\ZZ$ forces $\mu=\pm 12$, exactly.
\end{proof}
\begin{lem}\label{ints} On the basis $\{\sigma, N\sigma, N^2\sigma, \tau\}$ for $H_3(X_{\phi}(\CC), \QQ)$, the matrix of the intersection pairing is of the form 
$$\begin{pmatrix} 0 & -a & 0 & \pm 1\\a & 0 & -12 & 0\\0 &  12 & 0 & 0\\\mp 1 & 0 & 0 & 0\end{pmatrix}.$$
\end{lem}
\begin{proof} For any $x,y\in H_3(X_{\phi}(\CC), \QQ)$ we have 
$$\langle Nx,y\rangle = -\langle x, Ny\rangle,$$ as in \cite[Appendix A]{Mo}. Thus by Lemma \ref{12}
$$\langle N\sigma, \tau\rangle=\pm\frac{1}{12}\langle N\sigma, N^3\sigma\rangle=\mp\frac{1}{12}\langle\sigma, N^4\sigma\rangle=0,$$ since $N^4=0$. Similarly $\langle N^2\sigma, \tau\rangle =0$, and $\langle \sigma, N^2\sigma\rangle =-\langle N\sigma, N\sigma\rangle =0$, while 
$$\langle N\sigma, N^2\sigma\rangle=-\langle \sigma, N^3\sigma\rangle=-\langle \sigma, \pm 12\tau\rangle =-12.$$
\end{proof} 
If we apply $\iota_{\omega}$ to $\{\sigma, N\sigma, N^2\sigma, \tau\}$, we get $\pm\{v_0, v_1, 12v_2, v_3\}$, assuming that $\iota_{\omega}(\pm\sigma)=v_0:=-\frac{8\zeta(3)}{(2\pi i)^3}\Psi_0+(1/2)\Psi_1+12\Psi_3$, which is true at least to a very good approximation. Recall that according to Conjecture \ref{basis}, $\{v_0, v_1, v_2, v_3\}$ is (the image under $\iota_{\omega}$ of) a $\ZZ$-basis for $H_3(X_{\phi}(\CC), \ZZ)$.  The next proposition implies that if the approximation is exact then Conjecture \ref{basis} is true at least locally at $5$, in fact the same proof shows it would be true locally away from $2$ and $3$. Recall from Proposition \ref{sigm} that the numerical equality of $\iota_{\omega}(\pm\sigma)$ and $v_0$ comes from the computation of $T_{1/25}$, so the error may be bounded as in the proof of Lemma \ref{correct} below. 
\begin{prop}\label{Basis} The $\ZZ_{(5)}$-span of $\{\sigma, N\sigma, N^2\sigma, \tau\}$ is $H_3(X_{\phi}(\CC), \ZZ_{(5)})$.
\end{prop}
\begin{proof} Define $L$ to be the $\ZZ_{(5)}$-span of $\{\sigma, N\sigma, N^2\sigma, \tau\}$. By construction, this is a sublattice of $H_3(X_{\phi}(\CC), \ZZ_{(5)})$. Since $\ord_5(12)=0$ and the intersection pairing on $H_3(X_{\phi}(\CC), \ZZ)$ is unimodular, Lemma \ref{ints} shows that $L$ must already be the whole of $H_3(X_{\phi}(\CC), \ZZ_{(5)})$. 
\end{proof}
\begin{remar} Using computations of van Straten \cite{vS}, the matrix representing $T_{1/25}$ with respect to the basis $\{\sigma, N\sigma, N^2\sigma, \tau\}$ is 
$$\begin{pmatrix} 1&-10&0&\mp 10\\0&1&0&0\\0&0&1&0\\0&0&0&1\end{pmatrix}.$$
In particular, $N\sigma\mapsto N\sigma -10\sigma$. According to the Picard-Lefschetz formula, up on $\overline{X}_{\mathbf{a}}(\CC)$ we have $N\sigma'\mapsto N\sigma'+\langle N\sigma',\sigma'\rangle\sigma'$. Pushing forward to $X_{\phi}(\CC)$ and dividing by $10$, we find that $\langle N\sigma',\sigma'\rangle=-10$.
Since the homology class $\sigma'$ is $G$-fixed, its Poincar\'e dual is the pullback of that of $\sigma$. Applying $N$, the Poincar\'e dual of $N\sigma'$ is the pullback of that of $N\sigma$. Wedging together Poincar\'e duals and dividing by the fundamental class recovers intersection products, but the $10$-to-$1$ projection map sends the fundamental class of $X_{\phi}$ to $10$ times that of $\HV_{\mathbf{a}}$. It follows that $\langle N\sigma, \sigma\rangle=\frac{1}{10}\langle N\sigma',\sigma'\rangle=-1$. Hence $a=-1$, and the intersection matrix becomes
$$\begin{pmatrix} 0 & 1 & 0 & \pm 1\\-1 & 0 & -12 & 0\\0 & 12 & 0 & 0\\ \mp 1 & 0 & 0 & 0\end{pmatrix}.$$
Now choosing the basis $\{\sigma, N\sigma\mp\tau, \frac{1}{12}N^2\sigma, \tau\}$ will produce an intersection matrix $$\begin{pmatrix} 0 & 0 & 0 & \pm 1\\0 & 0 & -1 & 0\\0 &  1 & 0 & 0\\ \mp 1 & 0 & 0 & 0\end{pmatrix}.$$
\end{remar}
\begin{remar}\label{sgn} The basis in \cite[\S 3.2]{COES} (with third and fourth elements swapped) is  $\{\pm\sigma , \mp N\sigma +\tau, \frac{1}{12}N^2\sigma, \tau\}$. (Recall that $v_1=-v_1'+v_3$.) With respect to this basis, the intersection matrix becomes $$\begin{pmatrix} 0 & 0 & 0 & 1\\0 & 0 & \pm 1 & 0\\0 &  \mp 1 & 0 & 0\\ -1 & 0 & 0 & 0\end{pmatrix}.$$  But the monodromy matrices with respect to this basis, whose reductions mod $5$ appear in \S 3 below, will only preserve this symplectic form if it is in standard form, forcing $\pm 1=1$, i.e. $\langle\sigma, \tau\rangle =1$. (Check the matrix $B$, for instance.) \hfill$\triangle$
\end{remar}

\section{The image of monodromy mod $5$, and its normaliser}
Recall that $\Sigma:=\{0,1,\infty, \frac{1}{9}, \frac{1}{25}\}$. We have a smooth family $X\rightarrow \PP^1(\CC)\backslash\Sigma$ with fibres $X_{\phi}$. The integral cohomology modules $U_{\phi}:=H^3_B(X_{\phi}(\CC),\ZZ)/\{\mathrm{tors}\}\simeq H_3(X_{\phi}(\CC). \ZZ)/\{\mathrm{tors}\}$ form a local system over $\PP^1(\CC)\backslash\Sigma$, and fixing any basepoint $\phi$ we get a monodromy representation 
$$M: \pi_1(\PP^1(\CC)\backslash\Sigma)\rightarrow\Aut_{\ZZ}(U_{\phi}).$$
The fundamental group $\pi_1(\PP^1(\CC)\backslash\Sigma)$ is generated by loops
$c_0, c_1, c_{\infty}, c_{1/9}, c_{1/25}$ about the singular parameter values, satisfying a relation $c_0c_1c_{\infty}c_{1/9}c_{1/25}=\mathrm{id}$, so it suffices to consider $M(c_0)$, $M(c_1)$, $M(c_{1/9})$ and $M(c_{1/25})$. 

In \cite[\S 3.4]{COES} these are expressed as matrices $M_0, M_{1/25},\ldots$ with respect to the presumed basis $\{\pm\sigma, \mp N\sigma +\tau, \tau, \frac{1}{12}N^2\sigma\}$ for $H_3(X_{\phi}(\CC), \ZZ)$. What they actually do is to obtain (by the process outlined before Lemma \ref{T125}) numerical approximations to the monodromy matrices $R'_0, R'_{1/25},\ldots$ with respect to the basis $\{v_0, v'_1, v_2, v_3\}$ for $\ker(\mathcal{L})$, which by Proposition \ref{sigm} and the isomorphism $\iota_{\omega}$ are at least approximately the same as the matrices $R_0, R_{1/25},\ldots$ with respect to $\mathcal{B}:=\{\pm\sigma, \mp N\sigma +\tau, \tau, \frac{1}{12}N^2\sigma\}$, which by Proposition \ref{Basis} is a basis for $H_3(X_{\phi}(\CC), \ZZ_{(5)})$. The entries of the matrices $R'_p$ are extremely close to integers, and are then replaced by those integers to produce the matrices $M_0, M_{1/25},\ldots$.

\begin{lem}\label{correct}  These $M_0, M_{1/25},\ldots$ are indeed the correct monodromy matrices with respect to the basis $\mathcal{B}:=\{\pm\sigma, \mp N\sigma +\tau, \tau, \frac{1}{12}N^2\sigma\}$ for $H_3(X_{\phi}(\CC), \ZZ_{(5)})$. 
\end{lem}
\begin{proof} Let $Q_0, Q_{1/25},\ldots$ be the monodromy matrices with respect to an integral basis $\mathcal{B}'$ for $H_3(X_{\phi}(\CC), \ZZ)$ for which the intersection matrix is $$\begin{pmatrix} 0 & 0 & 0 & 1\\0 & 0 & \pm 1 & 0\\0 &  \mp 1 & 0 & 0\\ -1 & 0 & 0 & 0\end{pmatrix}.$$ We know that for each $p\in\Sigma$, $Q_p$ necessarily has integer entries. According to Conjecture \ref{basis} we could take $\mathcal{B}'=\mathcal{B}$, but we do not know this conjecture to be true. Let $D$ be a transition matrix from $\mathcal{B}'$ to $\mathcal{B}$, so for each $p\in \Sigma$ we have $R_p=DQ_pD^{-1}$. 

Using the proof of Lemma \ref{12}, specifically that $6N$ preserves $H_3(X_{\phi}(\CC), \ZZ)$, we see that the columns of $D$, in order, are in $\ZZ^4$, $\frac{1}{6}\ZZ^4$, $\ZZ^4$ and $\frac{1}{432}\ZZ^4$, respectively.
Hence the $3$-by-$3$ minors are all in $\frac{1}{2592}\ZZ$. Because $\mathcal{B}$ and $\mathcal{B}'$ have the same intersection matrix, $\det D=\pm 1$. Hence all the entries of $D^{-1}$ are in $\frac{1}{2592}\ZZ$, so all those of $R_p=DQ_pD^{-1}$ must be in $\frac{1}{1119744}\ZZ$.

 In the process of updating and relocating the AESZ database of Calabi-Yau operators (to \url{https://cycluster.mpim-bonn.mpg.de/} from \url{https://cydb.mathematik.uni-mainz.de/}), A. Klemm has recomputed the monodromy matrices $R'_p$, dealing with error bounds as in work of Mezzarobba \cite{Mez}, and kindly sharing his computations with me. The  $R'_p$ are equal to the integer matrices $M_p$ to a very high degree of accuracy, as are the $R_p$ to the $R'_p$, enough that the $R_p$, having entries in $\frac{1}{1119744}\ZZ$, are forced to be the same as the $M_p$. (Note that the error in the approximation of $R_p$ by $R'_p$ is related to the ``very good'' in Proposition \ref{sigm}, and may be bounded the same way, as noted just before Proposition \ref{Basis}.)
\end{proof}
Swapping the third and fourth basis elements, we can ensure that  the monodromy matrices lie in the symplectic group $\Sp_4(\ZZ)$ preserving the form $\begin{pmatrix}0&0&0&-1\\0&0&-1&0\\0&1&0&0\\1&0&0&0\end{pmatrix}$. (Recall from Remark \ref{sgn} that in fact $\pm =+$.) We are only concerned with their reductions mod $5$, giving the monodromy representation $\overline{M}$ of $\pi_1(\PP^1(\CC)\backslash\Sigma)$ on $U_{\phi}/5U_{\phi}$. These matrices lie in $\Sp_4(\FF_5)$. 

We read off from \cite[\S 3.4]{COES} that $\overline{M}(c_{1/25})=I_4$, the identity matrix. The others are 
$$A:=\overline{M}(c_0)=\begin{pmatrix}1&-1&\kappa &-2\kappa\\0&1&-2\kappa &-\kappa\\0&0&1&1\\0&0&0&1\end{pmatrix},\,\,B:=\overline{M}(c_{1/9})=\begin{pmatrix}1&-2&0&2\kappa\\0&1&0&0\\0&-2/\kappa &1&2\\0&0&0&1\end{pmatrix},$$ $$C:=\overline{M}(c_1)=\begin{pmatrix}1&-1&\kappa &\kappa\\0&0&\kappa &\kappa\\0&-1/\kappa &2&1\\0&0&0&1\end{pmatrix}.$$
Their $\kappa$ is $1$ or $2$, where $X_{\phi}$ has been obtained from $X_{\mathbf{a}}$ by quotienting by a cyclic group of order $10/\kappa$. In this paper we have chosen $\kappa=1$ for definiteness, but the case $\kappa=2$ may be dealt with similarly. (The choice affects $h^{1,1}=4\kappa +1$, but not the main theorem concerning $H^3(X_{\phi})$.)

Using the computer algebra system Maple \cite{M}, we find that 
$$A^2B^2C^3=\begin{pmatrix}1&4&3&2\\0&3&0&4\\0&0&2&2\\0&0&0&1\end{pmatrix},\,\,(AB^3)^3=\begin{pmatrix}1&0&0&1\\0&1&0&0\\0&0&1&0\\0&0&0&1\end{pmatrix},$$
$$(AB^2)^2(AB)^3(AB^3)^{12}=\begin{pmatrix}1&0&2&0\\0&1&0&2\\0&0&1&0\\0&0&0&1\end{pmatrix}$$ and $$((AB^2)^2(AB)^3)^2(AB)^3(BC)^2(AB^3)^6=\begin{pmatrix}1&4&0&0\\0&1&0&0\\0&0&1&1\\0&0&0&1\end{pmatrix}.$$
Let $G=\GSp_4(\FF_5)$ be the general symplectic group. It has a maximal parabolic subgroup $P$ with Levi subgroup
$L\simeq\GL_1(\FF_5)\times\GL_2(\FF_5)$ comprising matrices of the form
$$\begin{pmatrix}s&0&0&0\\0&\alpha &\beta &0\\0 &\gamma &\delta &0\\0&0&0&ts^{-1}\end{pmatrix},$$ with $A=\begin{pmatrix}\alpha &\beta\\\gamma &\delta\end{pmatrix}\in\GL_2(\FF_5)$ and $t=\det A$. The unipotent radical $U$ of $P$ comprises matrices of the form $$\begin{pmatrix}1&a&b&c\\0&1 &0 & b\\0&0 &1 & -a\\0&0&0&1\end{pmatrix}.$$
Using the above computations, it is easy to confirm the following.
\begin{prop}\label{mod5image}
The image $J:=\overline{M}(\pi_1(\PP^1(\CC)\backslash\Sigma))$ in $\Sp_4(\FF_5)$ is contained in $P$. It contains $U$, and its intersection with $L$ is the subgroup $K$ comprising matrices of the form 
$$\begin{pmatrix}1&0&0&0\\0&\alpha &\beta & 0\\0&\gamma &\delta & 0\\0&0&0&1\end{pmatrix},$$
with $\begin{pmatrix}\alpha &\beta\\\gamma &\delta\end{pmatrix}\in\SL_2(\FF_5)$.
\end{prop}
\begin{prop}\label{normaliser}
The normaliser $N_G(J)$ in $G$ of the monodromy image $J$ is contained in the parabolic subgroup $P$.  
\end{prop}
\begin{proof} Let $B$ be the upper-triangular Borel subgroup of $G$, containing the diagonal torus $T$ and contained in the maximal parabolic subgroup $P$. There is an $\FF_5$-rational Bruhat decomposition $G=\coprod_{w\in W}BwB$, where $W=N_G(T)/T$ is the Weyl group \cite[3.15(vi)]{DM}. Since clearly $B\subset N_G(J)$, any term $BwB$ is either entirely inside $N_G(J)$ or completely disjoint from it. The only non-identity element of $W$ normalising $J$ is $\omega=\begin{pmatrix}1&0&0&0\\0&0&-1&0\\0&1&0&0\\0&0&0&1\end{pmatrix}$, but $B\omega B\subset P$, in fact $P=B\cup B\omega B$.
\end{proof}
The \'etale $5$-adic cohomology $H^3_{\text{\'et}}(X_{-\frac{1}{7},\Qbar}, \Q_{5})$ contains a natural $G_{\Q}$-invariant lattice $H^3_{\text{\'et}}(X_{-\frac{1}{7},\Qbar}, \Z_{5})$, naturally identified with $U_{-\frac{1}{7}}\otimes\ZZ_5$. Let $\rho_5$ be the representation of $G_{\Q}$ on $H^3_{\text{\'et}}(X_{-\frac{1}{7},\Qbar}, \Z_{5})$, with $\rhobar_5$ the residual representation, on $U_{-\frac{1}{7}}/5U_{-\frac{1}{7}}$.
\begin{prop}\label{GaloisMod5}
\end{prop} The image of $\rhobar_5$ is contained in the maximal parabolic subgroup $P$.
\begin{proof} Choosing the rational point $\phi=-\frac{1}{7}$ as basepoint, we have the \'etale (algebraic) fundamental group $\pi_1^{\text{\'et}}(\PP^1_{\Q}\backslash\Sigma, -\frac{1}{7})$, acting on $U_{-\frac{1}{7}}\otimes\ZZ_5$ and thereby on $U_{-\frac{1}{7}}/5U_{-\frac{1}{7}}$. The \'etale fundamental group has a normal subgroup $\pi_1^{\text{\'et}}(\PP^1_{\Qbar}\backslash\Sigma, -\frac{1}{7})$ (with quotient $G_{\Q}$), naturally identified with the profinite completion of $\pi_1(\PP^1(\CC)\backslash\Sigma)$, and its action on $U_{-\frac{1}{7}}\otimes\ZZ_5$ with the monodromy representation. 

The rational basepoint $-\frac{1}{7}$ gives a section $G_{\Q}\hookrightarrow\pi_1^{\text{\'et}}(\PP^1_{\Q}\backslash\Sigma, -\frac{1}{7})$, which recovers the representations $\rho_5$ and $\rhobar_5$ of $G_{\Q}$ on $U_{-\frac{1}{7}}\otimes\ZZ_5$ and $U_{-\frac{1}{7}}/5U_{-\frac{1}{7}}$ respectively. The image of $G_{\Q}$ in $\pi_1^{\text{\'et}}(\PP^1_{\Q}\backslash\Sigma, -\frac{1}{7})$ normalises the normal subgroup $\pi_1^{\text{\'et}}(\PP^1_{\Qbar}-\Sigma, -\frac{1}{7})$. Therefore $\rhobar_5(G_{\Q})$ must normalise the monodromy image $J:=\overline{M}(\pi_1(\PP^1(\CC)\backslash\Sigma))$. Hence, by Proposition \ref{normaliser}, it is contained in $P$.
\end{proof}

\begin{remar} We should actually have used the transposes of the matrices in \cite{COES}, where matrices act on the right on row vectors, opposite to our convention. But combining transposition with a reversal of the basis is equivalent to reflection in the anti-diagonal, so our conclusions remain valid.\hfill$\triangle$
\end{remar}

\section{Counting points on the Hulek-Verrill manifold}
As already noted, the purpose of this section is to remove our dependence on the unproved conjecture in \cite[\S 4.4]{COS}. Recall that $X_{\phi}$ is a quotient, by a cyclic group of automorphisms of order $10$, of $\overline{X}_{\mathbf{a}}$, where $\mathbf{a}=(1:1:1:1:1:\phi^{-1})$.

\begin{lem}\label{HVcount} For any good prime $p$, 
$\#(\overline{X}_{\mathbf{a}}\backslash T)(\FF_p)=50p^2+40p+20$.
\end{lem}
\begin{proof} According to \cite[Lemma 6.4]{HV}, $\#(X_{\mathbf{a}}\backslash T)(\FF_p)=50p^2+10p+20$. Then in replacing the $30$ nodes by projective lines, another $30p$ points are added, as in the proof of \cite[Lemma 6.5]{HV}. 
\end{proof}

\begin{lem}\label{H2X}
For any prime $p$ of good reduction, the geometric Frobenius element $\Frob_p^{-1}\in G_{\Q}$ acts on $H^2_{\text{\'et}}(\overline{X}_{\mathbf{a},\Qbar}, \Q_{\ell})$ (for any prime $\ell$) as multiplication by $p$.
\end{lem}
\begin{proof} Let $P$ be a singular point on $X_{\mathbf{a}}$, for $\phi=-\frac{1}{7}$ necessarily outside the torus $T: X_1X_2X_3X_4X_5\neq 0$. In the blowup of $P$ in $X_{\mathbf{a}}$, the exceptional divisor above $P$ is isomorphic to a quadric surface $Q$ with two rulings. As in the proof of \cite[Proposition 6.1]{HV}, it suffices to show that (for each $P$) the rulings of $Q$ are defined over $\FF_p$. We look at just one of the $30$ nodes, since the others can be obtained from it by permutations of the variables.

Letting $x=X_1/X_2, y=X_2/X_3, z=X_3/X_4$ and $w=X_4/X_5$, as in \cite[(11)]{HV}, an affine piece of $X_{\mathbf{a}}$ is defined by the equation
$$(1+x+xy+xyz+xyzw)(1+w+wz+wzy+wzyx)=(-7)xyzw,$$
with node $P=(-1,0,0,-1)$. Using the Magma code in the supplementary file (working over $\Q$), we find that the exceptional divisor above $P$ in the blowup of $P$ in $X_{\mathbf{a}}$ is isomorphic to the quadric surface $Q$ in $\mathbb{P}^3$ with homogeneous equation 
$$-ac+ad+8bc-bd=0.$$
The rulings are defined over $\Q$ (and by reduction, over $\FF_p$), containing lines
$a=b=0$ and $c=d=0$.
\end{proof}

For the relationship between the small resolution that produces $\overline{X}_{\aaaa}$ from $X_{\aaaa}$ and the blowups of nodes referred to in the above proof, see \cite[\S 3.2.4]{HV}.

It follows, by Poincar\'e duality, that $\Frob_p^{-1}$ acts on $H^4_{\text{\'et}}(\HV_{\mathbf{a},\Qbar}, \Q_{\ell})$ as multiplication by $p^2$.

From now on, we specialise to $\phi=-1/7$. Following \cite[\S 5]{HV}, inside the torus $T: X_1X_2X_3X_4X_5\neq 0$ we may intersect $X_{\mathbf{a}}$ with the hyperplane $H_{45}: X_4+X_5=0$ to get the equation 
$$(X_1+X_2+X_3)\left(\frac{1}{X_1}+\frac{1}{X_2}+\frac{1}{X_3}\right)=\phi^{-1}=-7.$$
In the notation of \cite[\S 4]{HV}, this is the equation of an elliptic curve $\mathcal{E}_{1,1,1,-7}$.  By \cite[(21)]{HV}, this is isomorphic to an elliptic curve with Weierstrass equation $y^2=x(x^2+22x-7)$, which has minimal model
$y^2+xy+y=x^3-11x+12$, using Magma \cite{BCP}.  This is the curve labelled $\mathbf{14.a4}$ in the LMFDB database \cite{LMF}, whose $L$-series is $L(g,s)$, with $g$ as in Conjecture \ref{Conj1}.  Now let $E_{\aaaa}^{45}$ be the Zariski closure of $U_{\aaaa}^{45}:=X_{\aaaa}\cap T\cap H_{45}$ in $\overline{X}_{\aaaa}$ (not in $X_{\aaaa}$ as in \cite[\S 5]{HV}, but this does not affect the intersection matrices quoted below). Remembering the line given by $X_4+X_5=0$, we see that $E_{\aaaa}^{45}$ is birational to $\mathcal{E}_{1,1,1,-7}\times\mathbb{P}^1$, cf. \cite[Remark 5.2]{HV}. Permuting variables, we get $U_{\aaaa}^{ij}$ and $E_{\mathbf{a}}^{ij}$ for any pair $1\leq i\neq j\leq 5$.

\begin{lem}\label{smooth} For each $1\leq i\neq j\leq 5$, $E_{\aaaa}^{ij}$ is nonsingular, where $\mathbf{a}=(1:1:1:1:1:-7)$. 
\end{lem}
\begin{proof} This follows from \cite[Lemma 5.4]{HV}.
\end{proof}

For each $1\leq i\neq j\leq 5$ the closed embedding $r_{ij}: E_{\aaaa}^{ij}\hookrightarrow \overline{X}_{\aaaa}$ of nonsingular varieties gives rise to a Galois equivariant Gysin map $r_{ij, *}: H^1_{\text{\'et}}(E_{\aaaa,\Qbar}^{ij}, \Q_{\ell})(-1)\rightarrow H^3_{\text{\'et}}(\overline{X}_{\mathbf{a},\Qbar}, \Q_{\ell})$ \cite[VI, Remark 5.4]{Mi}. (Note that this is dual to the maps from $H^3_{\text{\'et}}(\overline{X}_{\mathbf{a},\Qbar}, \Q_{\ell})$ to $H^1_{\text{\'et}}(E_{\aaaa,\Qbar}^{ij}, \Q_{\ell})\times H^2_{\text{\'et}}(\PP^1_{\Qbar}, \Q_{\ell})$ appearing just below Corollary 5.11 in \cite{HV}.) Let $W_{\aaaa}$ be the subspace of $H^3_{\text{\'et}}(\overline{X}_{\mathbf{a},\Qbar}, \Q_{\ell})$ generated by the images of these maps, and let $W_{\aaaa}^{\perp}$ be its orthogonal complement with respect to the cup product pairing. The following is essentially part of \cite[Corollary 5.11]{HV}, but we take the opportunity to provide more details.

\begin{prop}\label{dimW}
The quotient $\frac{W_{\mathbf{a}}}{(W_{\mathbf{a}}\cap W_{\mathbf{a}}^{\perp})}$ is $8$-dimensional.
\end{prop}
\begin{proof} Since $H^3_{\text{\'et}}(\overline{X}_{\mathbf{a},\Qbar}, \Q_{\ell})\simeq H^3_B(\overline{X}_{\aaaa}(\C), \Q)\otimes_{\Q}\Q_{\ell})$ and $H^3_B(\overline{X}_{\aaaa}(\C), \Q)\otimes_{\Q}\C\simeq H^3_{\dR}(\overline{X}_{\aaaa}(\C), \C)$ (and likewise for $H^1(E_{\aaaa}^{ij})$), we may work with de Rham cohomology. The summands in the Hodge decomposition of the algebraic de Rham cohomology $H^1_{\dR}(E_{\aaaa,\C}^{ij})$ are $H^0(E_{\aaaa,\C}^{ij},\Omega^1)$ and $H^1(E_{\aaaa,\C}^{ij},\OOO_{E_{\aaaa}^{ij}})$. The former is a birational invariant, by \cite[II, Exercise 8.8]{Ha}, the latter by \cite[V, Proposition 3.4]{Ha}. 
Recall that $E_{\aaaa}^{ij}$ is birational to $\mathcal{E}_{1,1,1,-7}\times\mathbb{P}^1$, whose $1^{\mathrm{st}}$ cohomology is generated by classes $\alpha$ and $\beta$, each Poincar\'e dual to the product of the Riemann sphere and a loop on the elliptic curve. Restricting the de Rham cohomology to the open set $U_{\aaaa}^{ij}$ and using the birational invariance, we may view $\alpha$ and $\beta$ as generating elements of $H^1_{\dR}(E_{\aaaa}^{ij}(\C), \C)$. Let $\alpha^{ij}$ and $\beta^{ij}$ be the images of these classes in $H^3_{\dR}(\overline{X}_{\aaaa}(\C), \C)$ under the Gysin maps. 

Define matrices of cup products:
$$A:=\begin{pmatrix}\alpha^{ij}\cdot\alpha^{ij} & \alpha^{ij}\cdot\beta^{ij}\\\beta^{ij}\cdot\alpha^{ij} & \beta^{ij}\cdot\beta^{ij}\end{pmatrix}, B':=\begin{pmatrix}\alpha^{ij}\cdot\alpha^{ik} & \alpha^{ij}\cdot\beta^{ik}\\\beta^{ij}\cdot\alpha^{ik} & \beta^{ij}\cdot\beta^{ik}\end{pmatrix}, Z:=\begin{pmatrix}\alpha^{ij}\cdot\alpha^{kl} & \alpha^{ij}\cdot\beta^{kl}\\\beta^{ij}\cdot\alpha^{kl} & \beta^{ij}\cdot\beta^{kl}\end{pmatrix},$$
where $i,j,k,l$ are distinct elements of $\{1,2,3,4,5\}$. Lemmas 5.8, 5.10 and 5.9 of \cite{HV} give
$$A=\begin{pmatrix}0 & -2\\2 & 0\end{pmatrix}, B'=\pm B=\pm \begin{pmatrix} 0 & 1\\-1 & 0\end{pmatrix}, Z=0_2,$$ where $\pm$ is a certain sign depending on $(\{i,j\},\{k,l\})$. Although their $E_{\aaaa}^{ij}$ are Zariski closures of $U_{\aaaa}^{ij}$ in $X_{\aaaa}$, not $\overline{X}_{\aaaa}$, for $B'$ and $Z$, the arguments in \cite{HV}, using intersection numbers of $3$-cycles, work the same, since those intersections are transverse and the intersection points are inside $U_{\aaaa}^{ij}$.

Their proof for $A$ seems in fact to use our $E_{\aaaa}^{ij}$ (their $\tilde{E}_{\aaaa}^{ij}$), but since it is not easy to see how it follows from the quoted reference, and as a pretext for giving more detail about the rank $8$, we offer an alternative. We know at least that $A=\lambda B$, for some scalar $\lambda$, and we would like to show that $\lambda=-2$. 
(Note that by symmetry, $\lambda$ is independent of $\{i,j\}$.)
Consider the intersection matrix
\vskip10pt
\begin{tabular}{c|cccccccccc} & $E^{12}$ & $E^{13}$ & $E^{14}$ & $E^{15}$ & $E^{23}$ & $E^{24}$ & $E^{25}$ & $E^{34}$ & $E^{35}$ & $E^{45}$\\\hline $E^{12}$ & $A$ & $+B$ & $-B$ & $+B$ & $+B$ & $-B$ & $+B$ & $0$ & $0$ & $0$\\
$E^{13}$ & $+B$ & $A$ & $+B$ & $-B$ & $+B$ & $0$ & $0$ & $-B$ & $+B$ & $0$\\
$E^{14}$ & $-B$ & $+B$ & $A$ & $+B$ & $0$ & $+B$ & $0$ & $-B$ & $0$ & $+B$\\
$E^{15}$ & $+B$ & $-B$ & $+B$ & $A$ & $0$ & $0$ & $+B$ & $0$ & $-B$ & $+B$\\
$E^{23}$ & $+B$ & $+B$ & $0$ & $0$ & $A$ & $+B$ & $-B$ & $+B$ & $-B$ & $0$\\
$E^{24}$ & $-B$ & $0$ & $+B$ & $0$ & $+B$ & $A$ & $+B$ & $+B$ & $0$ & $-B$\\
$E^{25}$ & $+B$ & $0$ & $0$ & $+B$ & $-B$ & $+B$ & $A$ & $0$ & $+B$ & $-B$\\
$E^{34}$ & $0$ & $-B$ & $-B$ & $0$ & $+B$ & $+B$ & $0$ & $A$ & $+B$ & $+B$\\
$E^{35}$ & $0$ & $+B$ & $0$ & $-B$ & $-B$ & $0$ & $+B$ & $+B$ & $A$ & $+B$\\
$E^{45}$ & $0$ & $0$ & $+B$ & $+B$ & $0$ & $-B$ & $-B$ & $+B$ & $+B$ & $A$
\end{tabular}
\vskip10pt
Assuming that indeed $\lambda=-2$, for $2\leq i<j\leq 5$ column $E^{ij}$ is a linear combination of columns $E^{1i}$ and $E^{1j}$. So by a sequence of column operations, followed by the corresponding row operations, we may reduce the matrix to its upper left $8$-by-$8$ (i.e. $4$-by-$4$ in blocks) submatrix, which has rank $8$. The proposition would follow. If $\lambda\neq -2$, after the same operations the only difference would be that in the lower right submatrix ($6$-by-$6$ in blocks) we would see the non-zero $(\lambda +2)B$ all along the leading diagonal. This would force $W_{\aaaa}$ to have rank at least $12$, necessarily equal to $12$, which is the dimension of the whole of $H^3_{\dR}(\overline{X}_{\aaaa}(\C), \C)$, since $h^{2,1}(\overline{X}_{\aaaa})=5$, from \cite[Proposition 4.8]{HV}. But then the trace of $\Frob_p^{-1}$ on $H^3_{\text{\'et}}(\overline{X}_{\mathbf{a},\Qbar}, \Q_{\ell})$ would be $6pb_p$ (where $g=\sum b_nq^n$), which would contradict the point counts in the proof of Proposition \ref{count} below.
\end{proof}
 
\begin{prop} \label{count} For $\mathbf{a}=(1:1:1:1:1:\phi^{-1})=(1:1:1:1:1:-7)$, and any prime $\ell$, we have
$$\det(I-\Frob_p^{-1}|H^3_{\text{\'et}}(\overline{X}_{\mathbf{a},\Qbar}, \Q_{\ell})p^{-s})=(1-a_pp^{-s}+p^{3-2s})(1-pb_pp^{-s}+p^{3-2s})^5,$$ for $p=3,11,13,17,19,29,31$ and $113$. Here $f=\sum a_nq^n$, $g=\sum b_nq^n$ are as in Conjecture \ref{Conj1}.
\end{prop}
\begin{proof}  By Proposition \ref{dimW}, the quotient $\frac{W_{\mathbf{a}}}{(W_{\mathbf{a}}\cap W_{\mathbf{a}}^{\perp})}$ is $8$-dimensional, necessarily four copies of $H^1_{\text{\'et}}(\mathcal{E}_{1,1,1,-7,\Qbar}, \Q_{\ell})(-1)$. 
Hence
$$\det(I-\rho_{\ell}(\Frob_p^{-1}|H^3_{\text{\'et}}(\overline{X}_{\mathbf{a},\Qbar}, \Q_{\ell}))p^{-s})=F(p^{-s})(1-pb_pp^{-s}+p^{3-2s})^4,$$ where $F(X)$ is a polynomial of degree $4$, in fact by Poincar\'e duality
$$F(X)=(1-\gamma X)(1-(p^3/\gamma)X)(1-\delta X)(1-(p^3/\delta)X)$$ only depends on two ``unknowns'' $\gamma$ and $\delta$. To prove that $F(X)=(1-a_pX+p^3X^2)(1-pb_pX+p^3X^2)$, for a given prime $p$ of good reduction, it therefore suffices to check that 
\begin{equation}\label{trp}\tr(\Frob_p^{-1}|H^3_{\text{\'et}}(\overline{X}_{\mathbf{a},\Qbar}, \Q_{\ell}))=a_p+5pb_p\end{equation}
and \begin{equation}\label{trp2}\tr(\Frob_p^{-2}|H^3_{\text{\'et}}(\overline{X}_{\mathbf{a},\Qbar}, \Q_{\ell}))=(a_p^2-2p^3)+5p^2(b_p^2-2p).\end{equation}
Rearranging the formula
$$\#\overline{X}_{\mathbf{a}}(\FF_p)=\sum_{i=0}^6 (-1)^i\tr(\Frob_p^{-1}|H^i_{\text{\'et}}(\overline{X}_{\mathbf{a},\Qbar}, \Q_{\ell}),$$ using Lemma \ref{H2X} (with $h^{1,1}(\overline{X}_{\mathbf{a}})=45$ from \cite[end of \S 3.2.5]{HV}), we find that
$$\tr(\Frob_p^{-1}|H^3_{\text{\'et}}(\overline{X}_{\mathbf{a},\Qbar}, \Q_{\ell}))=1+p^3+45(p+p^2)-\#\overline{X}_{\mathbf{a}}(\FF_p),$$
and similarly 
$$\tr(\Frob_p^{-2}|H^3_{\text{\'et}}(\overline{X}_{\mathbf{a},\Qbar}, \Q_{\ell}))=1+p^6+45(p^2+p^4)-\#\overline{X}_{\mathbf{a}}(\FF_{p^2}).$$
Now $\#(\overline{X}_{\aaaa}\cap T)(\FF_p)$, by \cite[Lemma 6.3]{HV}, is 
$$\sum_{x,y,z\in \FF_p^{\times}}\left(\left(\frac{((1+x+y+z)\left(1+\frac{1}{x}+\frac{1}{y}+\frac{1}{z}\right)-1-(\phi^{-1}))^2-4(1)(\phi^{-1})}{p}\right)+1\right)$$
$$-2(p^2-3p+3).$$ Note that the sum counts solutions to the quadratic equation (\ref{XTquad}).
Adding in $\#(\overline{X}_{\mathbf{a}}\backslash T)(\FF_p)=50p^2+40p+20$ from Proposition \ref{HVcount}, we find that
$$\#\overline{X}_{\mathbf{a}}(\FF_p)=48p^2+46p+14$$ $$+\sum_{x,y,z\in \FF_p^{\times}}\left(\left(\frac{((1+x+y+z)\left(1+\frac{1}{x}+\frac{1}{y}+\frac{1}{z}\right)-1-(\phi^{-1}))^2-4(1)(\phi^{-1})}{p}\right)+1\right),$$
cf. \cite[Lemma 6.5]{HV}. One obtains a similar formula for $\#\overline{X}_{\mathbf{a}}(\FF_{p^2})$ by substituting $p^2$ for $p$, and replacing the Legendre symbol by the character of $\FF_{p^2}^{\times}$ that distinguishes squares. 

Thus we can check (\ref{trp}) and (\ref{trp2}), for the primes in question, by computing both sides (with $\phi^{-1}=-7$). The intermediate results are in the table below.  Verifying (\ref{trp}) was almost instantaneous in each case, but for $p=113$ it took a little under three weeks to compute $\#\overline{X}_{\mathbf{a}}(\FF_{p^2})$ in order to verify (\ref{trp2}). The Magma code for this is in the supplementary file. As hoped, we found that $(1+p^6+45(p^4+p^2))-\#(\overline{X}_{\mathbf{a}})(\FF_{p^2})$ (with $p=113$) came out to precisely $-13117460$, the number in the fifth column of the table.
\end{proof}
\vskip10pt
\begin{tabular}{|c|c|c|c|c|c|c|}
\hline $p$ & $b_p$ & $a_p$ & $a_p+5pb_p$ & $(a_p^2-2p^3)+5p^2(b_p^2-2p)$ & $\#\overline{X}_{\mathbf{a}}(\FF_p)$ & $\#\overline{X}_{\mathbf{a}}(\FF_{p^2})$\\\hline
$3$ & $-2$ & $8$ & $-22$ & $-80$ & $590$ & $4860$\\
$11$ & $0$ & $-28$ & $-28$ & $-15188$ & $7300$ & $2451040$\\
$13$ & $-4$ & $18$ & $-10$ & $-4070$ & $10630$ & $6132180$\\
$17$ & $6$ & $74$ & $584$ & $-1460$ & $18100$ & $27910480$\\
$19$ & $2$ & $80$ & $270$ & $-68688$ & $23690$ & $52995260$\\
$29$ & $-6$ & $190$ & $-680$ & $-105188$ & $64220$ & $626794000$\\
$31$ & $-4$ & $72$ & $-548$ & $-275428$ & $74980$ & $929380800$\\
$113$ & $6$ & $1378$ & $4768$ & $-13117460$ & $2017820$ & 2089302575920
\\\hline
\end{tabular}
\vskip10pt
\begin{prop}\label{3to113} For any prime $\ell$ we have
$$\det(I-\rho_{\ell}(\Frob_p^{-1}|H^3_{\text{\'et}}(\overline{X}_{-1/7,\Qbar}, \Q_{\ell}))p^{-s})=(1-a_pp^{-s}+p^{3-2s})(1-pb_pp^{-s}+p^{1-2s}),$$ for $p=3,11,13,17,19,29,31$ and $113$.
\end{prop}
\begin{proof} Letting $\mathbf{a}=(1:1:1:1:1:-7)$, $H^3_{\text{\'et}}(\overline{X}_{\mathbf{a},\Qbar},\Q_{\ell})$ is $12$-dimensional (since $h^{2,1}(\overline{X}_{\mathbf{a}})=5$, from \cite[\S 3.2.5]{HV}), with $H^3_{\text{\'et}}(X_{-1/7,\Qbar}, \Q_{\ell})$ the $4$-dimensional space of invariants under the group of automorphisms generated by $X_i\mapsto \frac{1}{X_i}$ and $X_i\mapsto X_{i+1}$ (reading the subscripts mod $5$). By the proof of Proposition \ref{count}, $\frac{W_{\mathbf{a}}}{(W_{\mathbf{a}}\cap W_{\mathbf{a}}^{\perp})}$ is $8$-dimensional. To show that these $4$- and $8$-dimensional contributions to $H^3_{\text{\'et}}(\overline{X}_{\mathbf{a},\Qbar},\Q_{\ell})$ are independent (therefore accounting for all of it), it suffices to check that any element of $W_{\mathbf{a}}$ fixed by the $5$-cycle $X_i\mapsto X_{i+1}$ is in $W_{\mathbf{a}}^{\perp}$. Then this proposition follows from the statement of Proposition \ref{count}.

The subspace $W_{\mathbf{a}}$ of $H^3_{\text{\'et}}(\overline{X}_{\mathbf{a},\Qbar},\Q_{\ell})$ is generated by elements denoted $\alpha^{ij}$ and $\beta^{ij}$ in the proof of Proposition \ref{dimW} above. These elements are not linearly independent, but still we can say that any element of $W_{\mathbf{a}}$ fixed by the $5$-cycle must be a linear combination of four elements $$C:=\alpha^{12}+\alpha^{23}+\alpha^{34}+\alpha^{45}+\alpha^{51}, D:=\beta^{12}+\beta^{23}+\beta^{34}+\beta^{45}+\beta^{51},$$
$$E:=\alpha^{13}+\alpha^{24}+\alpha^{35}+\alpha^{41}+\alpha^{52}, F:=\beta^{13}+\beta^{24}+\beta^{35}+\beta^{41}+\beta^{52}.$$
It is easy to check, using the intersection matrices in the proof of Proposition \ref{dimW} (\cite[Lemmas 5.8, 5.9, 5.10]{HV}), that each of $C,D,E$ and $F$ is orthogonal to each $\alpha^{ij}$ and $\beta^{ij}$. For example,
$$C\cdot\beta^{12}=(\alpha^{12}+\alpha^{23}+\alpha^{34}+\alpha^{45}+\alpha^{51})\cdot\beta^{12}=(-2)+1+0+0+1=0.$$
\end{proof}
\begin{remar}\label{MeVe} The conjecture of Meyer and Verrill \cite[p.157, \S 5.8]{Me}, that the $L$-function associated with $H^3_{\text{\'et}}(\overline{X}_{\mathbf{a},\Qbar}, \Q_{\ell})$ is $L(f,s)L(g,s-1)^5$, was based on numerical evidence from point counts, presumably similar to those above. It would be an easy consequence of our main Theorem \ref{main} (conditional on Proposition \ref{notirred}), given that the presence of the factor $L(g, s-1)^4$ is already known, as noted in the proof of Proposition \ref{count}.\hfill$\triangle$
\end{remar}
\begin{remar}\label{others} The other two apparent rank-$2$ attractor points $\phi=33\pm 8\sqrt{17}$, reciprocal to each other, units of norm $1$, lead to elliptic curves $\mathcal{E}_{1,1,1,\phi^{-1}}$ isomorphic to the curves with Weierstrass equations $$y^2=x(x^2+(494\mp 120\sqrt{17})x+(33\mp 8\sqrt{17})),$$ which have minimal models $$y^2+xy+\frac{1}{2}(3\mp\sqrt{17})y=x^3+\frac{1}{2}(1\mp\sqrt{17})x^2+\frac{1}{2}(-53\mp 11\sqrt{17})x+\frac{1}{2}(-117\mp 27\sqrt{17}).$$ These conjugate curves over $\Q(\sqrt{17})$ are $\Q$-curves, isogenous to their Galois conjugates (each other). One of them is $\mathbf{4.1}$-$\mathbf{a11}$ among elliptic curves over $\Q(\sqrt{17})$ in \cite{LMF}. The $\Gal(\Q(\sqrt{17})/\Q)$-conjugate $\ell$-adic representations of $G_{\Q(\sqrt{17})}$ attached to these two curves are isomorphic to each other, the restrictions of the $\ell$-adic representations of $G_{\Q}$ attached to the modular form $g=\mathbf{34.2.b.a}$ appearing in Conjecture \ref{Conj2}. Thus the construction of Hulek and Verrill puts a factor $L(\rho_{g,\ell}|_{G_{\Q(\sqrt{17})}},s-1)^4$ into the $L$-function of $\overline{X}_{\mathbf{a}}$.\hfill$\triangle$
\end{remar}
\begin{remar} It would follow from Theorem \ref{main} below (conditional on Proposition \ref{notirred}) that (for primes $p\neq 2,7$, and $\phi=-1/7$) the character sum
$$S:=\sum_{x,y,z\in \FF_p^{\times}}\left(\frac{((1+x+y+z)\left(1+\frac{1}{x}+\frac{1}{y}+\frac{1}{z}\right)-1-(\phi^{-1}))^2-4(1)(\phi^{-1})}{p}\right)$$ satisfies
$$S+(p-1)^3+(48p^2+46p+14)=\#\overline{X}_{\mathbf{a}}(\FF_p)=1+p^3+45(p^2+p)-(a_p+5pb_p),$$
which implies that $$S=-(a_p+5pb_p+4p+12).$$
Notice how the $45p^2$ on each side cancels, to avoid getting something of a larger order of magnitude than one would expect from a random walk of this many $\{\pm 1\}$-steps. (We know that $|a_p|$ and $|pb_p|$ are bounded by $2p^{3/2}$.)\hfill$\triangle$
\end{remar}
\begin{remar} There is a different construction of a Calabi-Yau compactification of $X_{\mathbf{a}}\cap T$.  As pointed out in \cite[Lemma 3.4]{HV}, by setting $X_6=-\sum_{i=1}^5 X_i$, $X_{\mathbf{a}}\cap T$ may be defined by two equations:
$$X_1+X_2+X_3+X_4+X_5+X_6=\frac{1}{X_1}+\frac{1}{X_2}+\frac{1}{X_3}+\frac{1}{X_4}+\frac{1}{X_5}+\frac{\phi^{-1}}{X_6}=0.$$ In \cite[\S 2.4, The Hulek-Verrill manifold $\mathrm{HV}$]{COKM}, Candelas et al. observe that the closure of this in an appropriate $5$-dimensional toric variety produces a nonsingular Calabi-Yau $3$-fold $\HV_{\aaaa}$, with no need for a small resolution step in the construction. Taking the quotient by the group $G$ of automorphisms (when $\aaaa=(1:1:1:1:1:\phi^{-1})$), they obtain a Calabi-Yau threefold with $h^{1,2}=h^{2,1}=1$, which we denote $\HV_{\phi}$. They raise the question of whether $\HV_{\aaaa}$ is isomorphic to $\overline{X}_{\aaaa}$. It is possible to prove the analogue of Lemma \ref{HVcount} for $\HV_{\aaaa}$, so at least they have the same number of points over $\FF_p$, for any prime $p$ of good reduction.\hfill$\triangle$
\end{remar}

\section{Identifying a $2$-dimensional mod $5$ Galois representation}
\begin{lem}\label{goodred}
\end{lem} The set of primes for which $\rho_5$ is ramified is contained in $\{2,5,7\}$. It is crystalline at $5$.
\begin{proof}
To identify the primes of bad reduction for $X_{-\frac{1}{7}}$, we just need to look for primes modulo which the parameter $-\frac{1}{7}$ becomes the same as one of the elements of $\Sigma=\{0,1,\infty, \frac{1}{9}, \frac{1}{25}\}$. The justification for this follows from the determination of singular parameters in \cite[(2.15)]{COKM}, which works over any field, as does \cite[Lemma 3.7]{HV} on which it is based, a fact noted in \cite[Lemma 6.2]{HV}. 

The parameter $-\frac{1}{7}$ coincides with $\infty\pmod{7}$, while $1-\left(-\frac{1}{7}\right)=\frac{2^3}{7}$, $\frac{1}{9}-\left(-\frac{1}{7}\right)=\frac{2^4}{63}$ and $\frac{1}{25}-\left(-\frac{1}{7}\right)=\frac{2^5}{175}$. So the primes of bad reduction are $2$ and $7$. It follows that $H^3_{\text{\'et}}(X_{-\frac{1}{7},\Qbar}, \Q_5)$ is unramified outside $\{2,5,7\}$, and is crystalline at $5$, the latter by a theorem of Fontaine and Messing \cite[Theorem A(ii)(a), Theorem B]{FM}, comparing crystalline cohomology and $p$-adic \'etale cohomology. The subspace $H^3_{\text{\'et}}(\HV_{-\frac{1}{7},\Qbar}, \Q_5)$ of invariants under the automorphism group of order $10$ has the same properties.   
\end{proof}

The residual $4$-dimensional representation $\rhobar_5$ has image contained in $P$.
Projecting to the Levi subgroup $L\simeq {\FF_5}^{\times}\times\GL_2(\FF_5)$ we obtain a character $\chi: G_{\Q}\rightarrow {\FF_5}^{\times}$ and a $2$-dimensional representation $\rho': G_{\Q}\rightarrow\GL_2(\FF_5)$.

As just noted, the $5$-adic Galois representation $\rho_5$ is crystalline. The set of Hodge-Tate weights is $\{0,1,2,3\}$. The prime $5$ is just big enough compared to the length of the Hodge filtration that the theory of Fontaine-Lafaille \cite{FL} applies, so the Hodge-Tate weights can be recovered from $\rhobar_5$. In particular, $\chi$ must be of the form
$\chi=\epsilon^a\psi$, where $\epsilon$ is the mod $5$ cyclotomic character, $a\in\{0,-1,-2,-3\}$ and $\psi$ is a finite-order character of conductor divisible at most by primes $2$ and $7$. (That a character to $\FF_5^{\times}$ of conductor divisible at most by $2,5$ and $7$ must be of this form is elementary, but it is useful to keep track of the Hodge-Tate weights for accounting purposes in the proof of Proposition \ref{Zp} below.)

The character $\chi=\epsilon^a\psi$ is among the composition factors of $\rhobar_5$. By Poincar\'e duality, so is $\epsilon^{-3}\chi^{-1}=\epsilon^{-3-a}\psi^{-1}$, so the possibilities for $\chi$ come in pairs. Given a prime $p\notin\{2,5,7\}$, let $Y_p:=\chi(\Frob_p^{-1})+\epsilon^{-3}\chi^{-1}(\Frob_p^{-1})$ and $Z_p:=\tr(\rho'(\Frob_p^{-1}))$. So $\tr(\rhobar_5(\Frob_p^{-1}))=Y_p+Z_p$.
Here $\Frob_p$ reduces mod $p$ to the $p^{\mathrm{th}}$-power automorphism of $\Fbar_p$.

\begin{prop}\label{Yp} For all primes $p\notin\{2,5,7\}$, $Y_p=1+p^3$, i.e. $\{\chi,\epsilon^{-3}\chi^{-1}\}=\{\mathrm{id}, \epsilon^{-3}\}$.
\end{prop}
\begin{proof}
The cyclotomic character is such that for any prime $p\neq 5$, $\epsilon(\Frob_p)=p\pmod{5}$. Identifying $\Frob_p$ with $p$, $\epsilon$ may be viewed as a homomorphism $\epsilon:\Q^{\times}\backslash\A_{\Q}^{\times}\rightarrow\FF_5^{\times}$, factoring through the identity map on $(\Z/5\Z)^{\times}$. Similarly, $\psi$ may be viewed as a homomorphism $\psi: \Q^{\times}\backslash\A_{\Q}^{\times}\rightarrow\FF_5^{\times}$, factoring through some $(\Z/2^b7^c\Z)^{\times}$.
Since $\#(\FF_5^{\times})=4$, $b\leq 4$ and $c\leq 1$. Moreover, $\psi$ must be a product $\psi=\psi_1^d\psi_2^e\psi_3^f$, where $\psi_1:(\Z/16\Z)^{\times}\rightarrow\FF_5^{\times}$ is defined by $\psi_1(3)=2$, $\psi(-1)=1$, $\psi_2:(\Z/4\Z)^{\times}\rightarrow\FF_5^{\times}$ by $\psi_2(-1)=-1$, and $\psi_3:(\Z/7\Z)^{\times}\rightarrow\FF_5^{\times}$ is the Legendre symbol. Here $0\leq d\leq 3$, $0\leq e\leq 1$ and $0\leq f\leq 1$.

 The polynomial $\det(I-\rho_5(\Frob_p)T)$ is always found to factor in $\ZZ[T]$ as $(1-\alpha(pT)+p^3T^2)(1-\beta T+p^3T^2)$ (with $\alpha$ and $\beta$ depending on $p$), for those $p$ for which it is computed in \cite{COES}. Inspecting coefficients of $T$ and $T^2$, we find that
$$Y_p+Z_p=\tr(\rhobar_5(\Frob_p^{-1}))\equiv\alpha p+\beta\pmod{5}$$
and $$Y_pZ_p\equiv p\alpha\beta\pmod{5}.$$
This determines the unordered set $\{Y_p, Z_p\}=\{\alpha p, \beta\}$, the set of roots of the quadratic $x^2-(\alpha p+\beta)x+p\alpha\beta$. Comparing with $Y_p=\psi(p^{-1})+p^3\psi(p)$ and $Y_p=p\psi(p^{-1})+p^2\psi(p)$ (without loss of generality we may suppose that $a=0$ or $-1$), for each $p$ we can hope to reject some of the possibilities for $(a,\psi)$. Using several $p$, we might eliminate all but the $(0,\mathrm{id})$ we want.

All the entries in the table below are in $\FF_5$, except for $p$. Those for $\alpha p$ and $\beta$ could be derived from \cite[Table 2]{COES}, depending upon the validity of their method of computation, i.e. the conjecture in \cite[\S 4.4]{COS}. Thanks to Proposition \ref{3to113} we know for certain that they are correct. (Note that before reduction mod $5$, $\alpha p$ may be distinguished from $\beta$ by its divisibility by $p$, for the $p$ we need.)
\vskip10pt
\begin{tabular}{|c|cc|ccc|ccc|}
\hline $p$ & $\alpha p$ & $\beta$ & $p$ & $p^2$ & $p^3$ & $\psi_1(p)$ & $\psi_2(p)$ & $\psi_3(p)$ \\\hline$31$ & $1$ & $2$ & $1$ & $1$ & $1$ & $1$ & $-1$ & $-1$ \\$113$ &$3$&$3$&$3$&$4$&$2$&$1$&$1$&$1$\\$29$&$1$&$0$&$4$&$1$&$4$&$2$&$1$&$1$\\$13$&$3$&$3$&$3$&$4$&$2$&$2$&$1$&$-1$\\$17$&$2$&$4$&$2$&$4$&$3$&$1$&$1$&$-1$\\\hline
\end{tabular}
\vskip10pt
For $p=31$, $Y_p=2(-1)^{e+f}$. This must be either $\alpha p=1$ or $\beta=2$, necessarily the latter, with $e+f$ even, i.e. $e=f$.

For $p=113$, either $Y_p=3$ (if $a=0$) or $Y_p=2$ (if $a=-1$). But $\alpha p=\beta=3$, so we deduce that $a=0$.

For $p=29$, $Y_p=2^{-d}-2^d$, which is $\alpha p=1\iff d=1$, $\beta=0\iff d=0$, so we eliminate the possibilities $d=2$ and $d=3$. 

If $p=13$ then $Y_p=2^{-d}(-1)^f+2(2^d(-1)^f)$. To make this equal to $3$, checking the cases $d=0$ and $d=1$ we find that $f=d$, so $d=e=f$. 

Finally, if $p=17$ then $Y_p=(-1)^f(1+3)$, which must be $\beta=4$, with $f=0$.
Hence $a=d=e=f=0$, as required.
\end{proof}
\begin{prop}\label{Zp} 
\begin{enumerate}
\item The $2$-dimensional representation $\rho'$ is irreducible.
\item Its trace at $\Frob_p^{-1}$, $Z_p=pb_p$, where $g=\sum b_nq^n$ is a normalised cuspidal newform for $\Gamma_0(14)$ as in Conjecture \ref{Conj1}.
\end{enumerate}
\end{prop}
\begin{proof}
(1). If $\rho'$ is reducible then, given that the Hodge-Tate weights $0$ and $3$ are already accounted for (in Proposition \ref{Yp}), its composition factors must be of the form $\{\epsilon^{-1}\psi, \epsilon^{-2}\psi^{-1}\}$, with $\psi=\psi_1^d\psi_2^e\psi_3^f$ as above. So $Z_p=p\psi(p^{-1})+p^2\psi(p)$. 
But for $p=113$ this would be $Z_p=2$, which does not match either of $\alpha p=\beta=3$.

(2). We now know that $\rho'$ is irreducible. We actually look at the Tate twist $\rho'(1)$, which has Hodge-Tate weights $0$ and $1$. It follows from the theorem of Khare-Wintenberger and Kisin \cite{KW1, KW2, Kis} (formerly Serre's Conjecture) that $\rho'(1)$ is the mod $5$ representation of $G_{\Q}$ attached to a cuspidal weight-$2$ newform $h=\sum c_nq^n$ for $\Gamma_0(N)$, for some $N$, which we may choose to be minimal. Our goal is to show that $N=14$. Since $S_2(\Gamma_0(14))$ is $1$-dimensional, this would imply that $h=g$ as in Conjecture \ref{Conj1}.
By Lemma \ref{goodred}, we know that $N$ (the Artin conductor of $\rho'(1)$) is divisible at most by the primes $2$ and $7$. Moreover, if $N=2^a7^b$ (resetting some notation), we have $a\leq 8$ and $b\leq 2$, by \cite[(4.8.8), end of \S 4.9]{Se}.

Given a space of weight-$2$ cuspforms $S_2(\Gamma_0(M))$, there is an associated space of cuspidal modular symbols $\mathbb{S}'_2(\Gamma_0(M))$, a certain free $\Z$-module of rank $2\dim(S_2(\Gamma_0(M)))$. We shall actually consider the $\FF_5$-vector space $\mathbb{S}_2(\Gamma_0(M))$ of cuspidal modular symbols with coefficients in $\FF_5$. 

Inside $\mathbb{S}_2(\Gamma_0(2^87^2))$, there is a $2\sigma_0(2^87^2/N)=2(8-a+1)(2-b+1)$-dimensional subspace $\mathbb{S}_h$ coming from the newform $h$, including ``old'' modular symbols depending on divisors of $2^87^2/N$. For each prime $p\notin\{2,5,7\}$, within the scope of \cite[Table 2]{COES}, it is always the case that $\beta\equiv Y_p=1+p^3\pmod{5}$, so necessarily $\alpha p\equiv Z_p\equiv pc_p\pmod{5}$. It follows that $\mathbb{S}_h$ is inside the intersection of the kernels of the Hecke operators $T_p-\alpha$ for these $p$. Within the scope of \cite[Table 2]{COES} always $\alpha=b_p$, so equally the $2(8-1+1)(2-1+1)=32$-dimensional space $\mathbb{S}_g$ associated with the newform $g=\sum b_nq^n\in S_2(\Gamma_0(14))$ lies inside this intersection of kernels.

The first few operators, written mod $5$, include $T_{11}$, $T_{13}-1$, $T_{17}-1$, $T_{19}-2$. For these values of $p$, we may rely on Proposition \ref{3to113} rather than \cite[Table 2]{COES}, so do not require the validity of the conjecture in \cite[\S 4.4]{COS}. We can find the intersection of the kernels of these operators using the Magma code in a supplementary file. 

The result (within less than a minute) is a $32$-dimensional space, which must be $\mathbb{S}_g$. Hence $\mathbb{S}_h\subseteq\mathbb{S}_g$. It could be just that $g$ and $h$ share the same mod $5$ residual Galois representation $\rho'(1)$, with their associated spaces of integral modular symbols overlapping when they are reduced mod $5$. But since $h$ was chosen to be of minimal level for that residual representation, in fact $h=g$, as required.
\end{proof}

\begin{thm}\label{chars} Let $F$ be any number field, and $\phi\in F$ a parameter value not in $\Sigma:=\{0,1,\infty, \frac{1}{9}, \frac{1}{25}\}$. Let $\rhobar_5^{\phi}$ be the
representation of $G_F$ on the reduction of $H^3_{\text{\'et}}(X_{\phi,\Qbar}, \Z_{5})$. Then $\mathrm{id}$ and $\epsilon^{-3}$ are among the composition factors of $\rhobar_5^{\phi}$. 
\end{thm}
\begin{proof}  Fixing a number field $F$, for any $\phi\in F-\Sigma$ there is a section $G_{F}\hookrightarrow\pi_1^{\text{\'et}}(\PP^1_{F}-\Sigma, \phi)$, whose image in $\GSp_4(\FF_5)$ lies in the maximal parabolic subgroup $P$ as in Proposition \ref{GaloisMod5}. Different choices of $\phi$ (for the same $F$) give sections differing only up to elements of the geometric monodromy group $\pi_1^{\text{\'et}}(\PP^1_{\Qbar}-\Sigma, \phi)$. (For different values $\phi_1, \phi_2$ we may choose a path from $\phi_1$ to $\phi_2$, which serves both to fix an isomorphism of geometric monodromy groups and to identify the fibres $U_{\phi_1}$ and $U_{\phi_2}$ on which they act.) Projecting the geometric monodromy action to the factor $\GL_1(\FF_5)$ of the Levi subgroup of $P$ gives the trivial character, as we can see from the $1$'s in the corners in Proposition \ref{mod5image}. It follows that the character $\chi$ is independent of $\phi$, hence $\{\chi, \epsilon^{-3}\chi^{-1}\}$ is $\{\mathrm{id}, \epsilon^{-3}\}$ for any $\phi\in F$, as it is for $\phi=-\frac{1}{7}$ by Proposition \ref{Yp}.
\end{proof}

\section{Plus/minus $p$-adic $L$-functions and Selmer groups at a supersingular prime}
Let $E/\Q$ be an elliptic curve with good reduction at an odd prime $p$. We suppose that $a_p(E)=0$, where $\#E(\FF_p)=1+p-a_p(E)$. If $p\geq 5$, this is equivalent to $E$ having supersingular reduction at $p$, i.e. $p\mid a_p(E)$. Both eigenvalues of Frobenius $\alpha=\pm\sqrt{p}\,i$ satisfy the condition $\ord_p(\alpha)<k-1$, where $k=2$ is the weight of the associated modular form. (In this section, the meaning of the notation $\alpha$ is different from the rest of the paper.) Consequently, for either choice there is a $p$-adic $L$-function $L_p(E,\alpha,T)\in \Q_p(\alpha)[[T]]$ \cite{AV, V}, but unlike the case of ordinary reduction, the coefficients of the power series are not bounded.

Let $\eta:(\Z/p\Z)^{\times}\rightarrow \Q(\zeta_{p-1})^{\times}$ be a Dirichlet character. Fixing embeddings from $\Q(\zeta_{p-1})$ to $\Q_p$ and to $\C$, $\eta$ may be viewed as taking values in $\Z_p^{\times}$ or in $\C^{\times}$. Identifying $(\Z/p\Z)^{\times}$ with $\Gal(\Q(\zeta_p)/\Q)$, it may also be thought of as a character of $G_{\Q}$. More generally, there are $p$-adic $L$-functions $L_p(E,\eta,\alpha,T)\in \Q_p(\alpha)[[T]]$ \cite{AV, V}.

Kobayashi \cite[Theorems 3.2, 6.3, end of \S 8]{Ko} proved the existence of $L_p^{\pm}(E,\eta, T)\in\ZZ_p[[T]]$ such that $$L_p(E,\eta, \alpha, T)=L_p^-(E,\eta,T)\cdot\log_p^+(T)+L_p^+(E,\eta,T)\cdot\log_p^-(T)\cdot\alpha.$$ Here $\log_p^{\pm}(T)\in\Q_p[[T]]$ are certain power series with zeroes at $\zeta_{p^n}-1$, where $\zeta_{p^n}$ runs through all $p$-power roots of unity (even powers for $\log^+(T)$, odd powers for $\log^-(T)$). We shall need the following small part of the interpolation properties of $L_p^{\pm}(E,\eta,T)$.

If $\eta$ is non-trivial and even (i.e. $\eta(-1)=1$) then
\begin{equation}\label{interp} L_p^+(E,\eta,0)=-\frac{p}{\tau(\overline{\eta})}\,\frac{L(E,\overline{\eta},1)}{\Omega_E},\end{equation}
where $\tau(\overline{\eta})$ is a Gauss sum and $\Omega_E$ is the real period of a N\'eron differential.
 
Kobayashi credits the theorem to Pollack, but gives a new proof. Actually the theorem stated by Pollack \cite[Theorem 5.6]{P} only covers trivial $\eta$, whereas our application requires a non-trivial $\eta$.

Now let $K_{-1}:=\Q$, and for $n\geq 0$ let $K_n:=\Q(\zeta_{p^{n+1}})$, with $K_{\infty}:=\cup K_n$. For each $n$, the $p$-primary Selmer group is defined by 
$$\mathrm{Sel}(E/K_n):=\Ker\left(H^1(K_n, E[p^{\infty}])\xrightarrow{\mathrm{res}}\prod_v\,\frac{H^1(K_{n,v}, E[p^{\infty}])}{E(K_{n,v})\otimes \Q_p/\Z_p}\right),$$
where $v$ runs over all places of $K_n$, and $\mathrm{Sel}(E/K_{\infty}):=\varinjlim \mathrm{Sel}(E/K_n)$. Then following Kobayashi \cite[Definition 2.1]{Ko} we define
$$E^+(K_{n,v}):=\{P\in E(K_{n,v})\,\mid\,\Tr_{K_{n,v}/K_{m+1,v}}P\in E(K_{m,v})\,\text{for all even $m$ }(0\leq m<n)\},$$
$$E^-(K_{n,v}):=\{P\in E(K_{n,v})\,\mid\,\Tr_{K_{n,v}/K_{m+1,v}}P\in E(K_{m,v})\,\text{for all odd $m$ }(-1\leq m<n)\},$$ 
$$\mathrm{Sel}^{\pm}(E/K_n):=\Ker\left(H^1(K_n, E[p^{\infty}])\xrightarrow{\mathrm{res}}\prod_v\,\frac{H^1(K_{n,v}, E[p^{\infty}])}{E^{\pm}(K_{n,v})\otimes \Q_p/\Z_p}\right),$$ and $\mathrm{Sel}^{\pm}(E/K_{\infty}):=\varinjlim \mathrm{Sel}^{\pm}(E/K_n)$. (The word ``all'' does not appear in Kobayashi's definition of $E^{\pm}(K_{n,v})$, but I have taken it from \cite[Definition 2.1]{KO}.)
The Pontryagin duals $\Hom(\mathrm{Sel}^{\pm}(E/K_n), \Q_p/\Z_p)$ and $\Hom(\mathrm{Sel}^{\pm}(E/K_{\infty}), \Q_p/\Z_p)$ are denoted $\mathfrak{X}^{\pm}(E/K_n)$ and $\mathfrak{X}^{\pm}(E/K_{\infty})$, respectively.

Let $\mathcal{G}_n:=\Gal(K_n/\Q)$, $\mathcal{G}_{\infty}:=\Gal(K_{\infty}/\Q)$, $\Lambda_n:=\Z_p[\mathcal{G}_n]$ and $\Lambda:=\Z_p[[\mathcal{G}_{\infty}]]$. Then $\Lambda_n$ and $\Lambda$ act on $\mathfrak{X}^{\pm}(E/K_n)$ and $\mathfrak{X}^{\pm}(E/K_{\infty})$, respectively. Note that $\mathcal{G}_{\infty}\simeq \Delta\times\Gamma$, with $\Delta\simeq (\Z/p\Z)^{\times}$ and $\Gamma\simeq\Z_p$. Choosing a topological generator $\gamma$ of $\Gamma$, and sending $\gamma\mapsto 1+T$ identifies $\Z_p[[\Gamma]]$ with $\Z_p[[T]]$ and $\Lambda$ with $\Z_p[\Delta][[T]]$.

The following is \cite[Theorem 2.2]{Ko}.
\begin{lem}\label{LambdaTorsion}
 $\mathfrak{X}^{\pm}(E/K_{\infty})$ are finitely generated $\Lambda$-torsion modules.
\end{lem}
Now a theorem of Kitajima and Otsuki \cite[Main Theorem 1.3]{KO} gives
\begin{lem}\label{NoFinite}
Neither $\mathfrak{X}^{\pm}(E/K_{\infty})$ has any non-trivial, finite $\Z_p[[T]]$-submodule.
\end{lem}
Note that since $p$ is ramified in $K_0=\Q(\zeta_p)/\Q$, the earlier result of B. D. Kim \cite{Kim} is not strong enough for our purposes. It follows that for each character $\eta$ of $\Delta$, the $\eta$-component $\mathfrak{X}^{\pm}(E/K_{\infty})^{\eta}$ is isomorphic as a $\Z_p[[T]]$-module to $\Z_p[[T]]/\mathrm{Char}(\mathfrak{X}^{\pm}(E/K_{\infty})^{\eta})$, the quotient by its ``characteristic ideal''. Kobayashi states even and odd main conjectures, and uses Kato's Euler system work \cite{Ka} to prove divisibilities towards these conjectures. What we need is the following, part of \cite[Theorem 4.1]{Ko}.
\begin{lem}\label{MCdiv}
If $\eta$ is even and non-trivial, and the representation of $G_{\Q}$ on $E[p]$ has image $\GL_2(\FF_p)$, then $\mathrm{Char}(\mathfrak{X}^{+}(E/K_{\infty})^{\eta})$ divides $L_p^+(E,\eta,T)$.
\end{lem}
\begin{prop}\label{SelTriv} Suppose that $\ord_p\left(\frac{p}{\tau(\overline{\eta})}\,\frac{L(E,\overline{\eta},1)}{\Omega_E}\right)=0$, with $\eta$ even and non-trivial. Suppose also that the representation of $G_{\Q}$ on $E[p]$ has image $\GL_2(\FF_p)$. Then $\mathrm{Sel}(E/K_0)^{\eta}$ is trivial. 
\end{prop}
\begin{proof} First, it follows directly from the definition that $\mathrm{Sel}(E/K_0)^{\eta}$ is the same as  $\mathrm{Sel}^+(E/K_0)^{\eta}$. Taking the Pontryagin dual, it suffices to show that $\mathfrak{X}^{+}(E/K_0)^{\eta}$ is trivial.
The first part of \cite[Theorem 9.3]{Ko} (with $\omega_0^+=T$) gives a surjection from $\mathfrak{X}^{+}(E/K_{\infty})^{\eta}/T\,\mathfrak{X}^{+}(E/K_{\infty})^{\eta}$ to $\mathfrak{X}^{+}(E/K_0)^{\eta}/T\,\mathfrak{X}^{+}(E/K_0)^{\eta}$, which is the same as $\mathfrak{X}^{+}(E/K_0)^{\eta}$. Hence it suffices to show that $\mathfrak{X}^{+}(E/K_{\infty})^{\eta}/T\,\mathfrak{X}^{+}(E/K_{\infty})^{\eta}$ is trivial. 

By Lemmas \ref{LambdaTorsion} and \ref{NoFinite}, $$\mathfrak{X}^{+}(E/K_{\infty})^{\eta}\simeq \Z_p[[T]]/\mathrm{Char}(\mathfrak{X}^{+}(E/K_{\infty})^{\eta}),$$ which by Lemma \ref{MCdiv} is a quotient of $$\mathfrak{X}^{+}(E/K_{\infty})^{\eta}\simeq \Z_p[[T]]/L_p^+(E,\eta,T).$$ Hence $\mathfrak{X}^{+}(E/K_{\infty})^{\eta}/T\,\mathfrak{X}^{+}(E/K_{\infty})^{\eta}$ is a quotient of $$\mathfrak{X}^{+}(E/K_{\infty})^{\eta}\simeq \Z_p[[T]]/(T, L_p^+(E,\eta,T))\simeq \Z_p/L_p^+(E,\eta,0),$$ which is trivial since 
$$\ord_p(L_p^+(E,\eta,0))=\ord_p\left(\frac{p}{\tau(\overline{\eta})}\,\frac{L(E,\overline{\eta},1)}{\Omega_E}\right)=0.$$
\end{proof}

\section{Constructing an element in a Selmer group}
\begin{prop}\label{notirred} The $4$-dimensional representation $\rho_5$ of $G_{\Q}$ on $H^3_{\mathrm{et}}(X_{-\frac{1}{7},\Qbar}, \Q_{5})$ is reducible. {\bf The proof of this is incorrect.}
\end{prop}
\begin{proof} Suppose to the contrary that $\rho_5$ is irreducible. Let $V'$ be the $\Q_5$-vector space $H^3_{\mathrm{et}}(X_{-\frac{1}{7},\Qbar}, \Q_{5})$ on which $G_{\Q}$ acts via $\rho_5$. Choosing the $G_{\Q}$-invariant $\Z_5$-lattice $T':=H^3_{\text{\'et}}(X_{-\frac{1}{7},\Qbar}, \Z_{5})$ and reducing mod $5$ gives the representation we have called $\rhobar_5$, on a space $\overline{T'}$, which by Propositions \ref{Yp} and \ref{Zp} has irreducible composition factors $\mathrm{id}$, $\epsilon^{-3}$ and $\rhobar_{g,5}(-1)$. The $2$-dimensional factor is the middle subfactor in the filtration of $\overline{T'}$ associated with the maximal parabolic subgroup $P$ in Proposition \ref{normaliser}.

Had $\rhobar_5$ been irreducible, the invariant lattice $T'$ would have been unique up to scaling. But because it is reducible, there are essentially different choices for $T'$. Arguing exactly as in the proof of \cite[Proposition 8.3]{Br}, which adapts an argument of Ribet \cite{Ri}, it is possible to choose a $G_{\Q}$-invariant lattice $T'$ in such a way that inside $\overline{T'}$ there is a non-split extension of $\mathrm{id}$ by either $\epsilon^{-3}$ or $\rhobar_{g,5}(-1)$. The former would produce a non-zero element in the $\epsilon^{-3}$-part of the $5$-part of the class group of $\Q(\zeta_5)$, which is however trivial, so it must be the latter. {\bf This is the error (which was revealed by Remark \ref{644} below). There is no reason why we cannot have a non-split extension of $\mathrm{id}$ by $\epsilon^{-3}$, since it is allowed to be ramified at the bad primes $2$ and $7$, which satisfy $2^4\equiv 7^4\equiv 1\pmod{5}$. Indeed the mod $5$ representation of $G_{\Q}$ attached to the weight-$4$ modular form $\mathbf{14.4.a.a}$ appearing in Conjecture \ref{Conj1} is just such an extension (for a suitable choice of invariant $\Z_5$-lattice before reduction mod $5$). Moreover, closer examination of \cite[Proposition 8.3]{Br} reveals that we would have proved not just reducibility, but existence of a $1$-dimensional subfactor, which would contradict the proof of Theorem \ref{main} below.} (It is perhaps worth mentioning that in the proof of \cite[Proposition 8.3]{Br}, the sums start at $n=0$, and $U_0$ does not have to be the same at each step of the induction.)

This produces a non-zero element $c$ in $H^1(\Q, \rhobar_{g,5}(-1))$. Now let $V$ be the $\Q_5$-vector space on which $\rho_{g,5}(-1)$ acts, with $G_{\Q}$-invariant lattice $T$ and $W:=V/T$, so $\rhobar_{g,5}(-1)$ is on the $5$-torsion in $W$. Following Bloch and Kato \cite{BK}, 
$$H^1_f(\Q_p, V):=\ker(H^1(\Q_p, V)\rightarrow H^1(I_p, V)\,\text{ for $p\neq 5$};$$
$$H^1_f(\Q_5, V):=\ker(H^1(\Q_p, V)\rightarrow H^1(\Q_p, V\otimes B_{\crys}),$$
where $I_p$ is the inertia subgroup of $G_{\Q_p}$ and $B_{\crys}$ is Fontaine's ring, whose construction is explained in \cite[\S 1]{BK}. The image of $H^1_f(\Q_p, V)$ in $H^1(\Q_p, W)$, and its inverse image in $H^1(\Q_p, T)$, are denoted $H^1_f(\Q_p, W)$ and $H^1_f(\Q_p, T)$, respectively. Let $S$ be a finite set of prime numbers, not including $5$. The Bloch-Kato Selmer group is defined to be
$$H^1_f(\Q, W):=\ker\left(H^1(\Q, W)\xrightarrow{\mathrm{res}}\bigoplus_p\frac{H^1(\Q_p, W)}{H^1_f(\Q_p, W)}\right),$$ and relaxing local conditions at primes in $S$,
$$H^1_S(\Q, W):=\ker\left(H^1(\Q, W)\xrightarrow{\mathrm{res}}\bigoplus_{p\notin S}\frac{H^1(\Q_p, W)}{H^1_f(\Q_p, W)}\right).$$

Since $\rhobar_{g,5}(-1)$ is irreducible, $H^0(\Q, W)$ is trivial, and hence the image  $d$ in $H^1(\Q, W)$ of the non-zero element $c\in H^1(\Q, \rhobar_{g,5}(-1))$ is non-zero. By Lemma \ref{goodred}, $\rho_5$ is crystalline at $5$ and unramified (trivial restriction to $I_p$) at all $p\notin\{2,5,7\}$. It follows, as in the proof of \cite[Proposition 8.3]{Br}, that $d\in H^1_{\{2,7\}}(\Q, W)$.

The cokernel of the inclusion $H^1_f(\Q, W)\hookrightarrow H^1_{\{2,7\}}(\Q, W)$ is contained in $\bigoplus_{p\in\{2,7\}}\frac{H^1(\Q_p, W)}{H^1_f(\Q_p, W)}$, which by local Tate duality is dual to $\bigoplus_{p\in\{2,7\}}H^1_f(\Q_p, T(4))$.
(Since $\Hom(\rho_{g,5}, \Q_5(1))\simeq \rho_{g,5}(2)$, $\Hom(\rho_{g,5}(-1), \Q_5(1))\simeq \rho_{g,5}(3)=\rho_{g,5}(-1)(4)$.) We shall show that $d\in H^1_f(\Q, W)$ by checking that $\bigoplus_{p\in\{2,7\}}H^1_f(\Q_p, T(4))$ is trivial.

First, $H^1_f(\Q_p, V(4))\simeq H^1(G_{\FF_p}, V(4)^{I_p})$ (by inflation-restriction). Since $G_{\FF_p}$ is pro-cyclic, generated by $\Frob_p$, 
$$H^1(G_{\FF_p}, V(4)^{I_p})\simeq \frac{V(4)^{I_p}}{(\Frob_p-1)V(4)^{I_p}}.$$
The elliptic curve $E/\Q$ associated with the newform $g$ (let's choose $\mathbf{14.a1}$ \cite{LMF} from the isogeny class, $E:y^2+xy+y=x^3-2731x-55146$) has non-split multiplicative reduction at $2$, split multiplicative reduction at $7$. It follows that $\Frob_p$ acts on the $1$-dimensional space(s) $V(4)^{I_p}$ as $-p^3$ when $p=2$, or as $p^3$ when $p=7$, so $H^1_f(\Q_p, V(4))$ is trivial, and $$H^1_f(\Q_p, T(4))\simeq H^1(\Q_p, T(4))_{\mathrm{tors}}\simeq H^0(\Q_p, W(4)).$$ Neither Tamagawa factor, at $2$ or at $7$, is divisible by $5$, so $W^{I_p}=V^{I_p}/T^{I_p}$. Plugging in $s=3$ to the Euler factors for $\rho_{g,5}$, 
$$\ord_5(\# H^0(\Q_2, W(4)))=\ord_5(1+2^{-3})=0,$$ $$\ord_5(\# H^0(\Q_7, W(4)))=\ord_5(1-7^{-3})=0.$$

Now we know we have a non-zero $5$-torsion element $d\in H^1_f(\Q, W)$. Since $V_5(E)\simeq \rho_5(1)$, $W\simeq E[5^{\infty}](-2)$, so $d\in H^1_f(\Q, E[5^{\infty}](-2))$. Recall that $c$ is $d$ viewed as an element of $H^1(\Q, E[5](-2))$. Mod $5$, the cyclotomic and Teichmuller characters, $\tilde{\epsilon}$ and $\omega$, are the same, so $c\in H^1(\Q, E[5]\otimes \eta)$, where $\eta=\omega^{-2}$, which is the Legendre symbol $\left(\frac{\cdot}{5}\right)$, the Kronecker character for $\Q(\sqrt{5})$. Note that $\eta$ is an even character of $\Delta=\Gal(K_0/\Q)$, where $K_0=\Q(\zeta_5)$, as in the previous section. 

Letting $d_0$ denote the restriction of $d$ to $H^1(K_0, E[5^{\infty}](-2))$, it is a direct consequence of the definitions that $d_0\in H^1_f(K_0, E[5^{\infty}](-2))$. Let $c_0$ be $d_0$ viewed as an element of $H^1(K_0, E[5](-2))$. Then $c_0$ is the restriction of $c\in H^1(\Q, E[5]\otimes \eta)$, from which it follows that $c_0\in H^1(K_0, E[5])^{\eta}$, hence $d_0\in H^1_f(K_0, E[5^{\infty}])^{\eta}$. By \cite[Example 3.11]{BK}, $d_0\in \mathrm{Sel}(E/K_0)^{\eta}$. 

Using LMFDB, we check that $E$ is supersingular at $p=5$, and that the representation of $G_{\Q}$ on $E[5]$ has image $\GL_2(\FF_5)$. Hence Proposition \ref{SelTriv} will give us the contradiction we desire, if we can just check that 
$\ord_5\left(\frac{p}{\tau(\overline{\eta})}\,\frac{L(E,\overline{\eta},1)}{\Omega_E}\right)=0$. But this is $\ord_5\left(\frac{L(E_5,1)}{\Omega_{E_5}}\right)$, where $E_5$ is the quadratic twist of $E$ by $\left(\frac{\cdot}{5}\right)$, which is $\mathbf{350.f1}$, and LMFDB tells us that $\frac{L(E_5,1)}{\Omega_{E_5}}=9$.                                                                                                                                  
\end{proof}
\begin{remar} In \cite[Appendix D]{BEKK} there is a table listing examples of experimental rank-$2$ attractor points at rational parameter values for various ``Calabi-Yau operators''. In ten cases (appearing in the Master's thesis of B\"onisch \cite[Table 7.1]{Bo}) there is an involutional symmetry, which should split the $4$-dimensional $\ell$-adic Galois representation into two $2$-dimensional pieces. In the remaining twelve cases (excluding that dealt with in this paper), one might wonder whether our method could work. Unfortunately, the answer appears to be ``No''. For the candidate weight-$2$ and weight-$4$ forms in  \cite[Appendix D]{BEKK}, their $2$-dimensional mod $\ell$ Galois representations are only ever reducible when $\ell=2$ or $3$. (One checks for congruences $a_p\equiv 1+p^3\pmod{\ell}$ or $b_p\equiv 1+p\pmod{\ell}$, for primes $p\nmid\ell N$.) Moreover, this only happens in cases where $\ell$ also divides the level, whereas for our construction of $d\in H^1_f(\Q, W)$ it is important that $\ell$ (equal to $5$ in our case) is a prime of good reduction, so that the local condition at $\ell$ is satisfied.\hfill$\triangle$
\end{remar}
\section{The final result}
\begin{thm}\label{main}
Conjecture \ref{Conj1} would follow from Proposition \ref{notirred}. In other words, the (semi-simplification of the) $4$-dimensional representation of $G_{\Q}$ on $H^3_{\text{\'et}}(X_{-\frac{1}{7},\Qbar}, \Q_{\ell})$ would be isomorphic to a direct sum $\rho_{f,\ell}\oplus \rho_{g,\ell}(-1)$, where $f$ and $g$ are cuspidal newforms of weights $4$ and $2$, respectively, for $\Gamma_0(14)$, labelled $\mathbf{14.4.a.a}$ and $\mathbf{14.2.a.a}$ in the LMFDB database \cite{LMF}.
\end{thm}
\begin{proof} As already noted in the Introduction, it suffices to prove this for $\ell=5$, i.e. for the $4$-dimensional representation we have called $\rho_5$. 
In Propositions \ref{Yp} and \ref{Zp} we have established that the irreducible composition factors of the residual representation $\rhobar_5$ are $\mathrm{id}$, $\epsilon^{-3}$ and $\rhobar_{g,5}(-1)$. 

First we show that $\rho_5$ has no $1$-dimensional composition factor. The associated character of $G_{\Q}$ would reduce mod $5$ to $\mathrm{id}$ or $\epsilon^{-3}$. It could only be $\mathrm{id}$ or $\tilde{\epsilon}^{-3}$, where $\tilde{\epsilon}$ is the $5$-adic cyclotomic character, there being no suitable character in the kernel of reduction mod $5$. This is because there are no non-trivial characters, of conductor divisible at most by the primes $2$ and $7$, taking $5^{\mathrm{th}}$ roots of $1$ as values, since neither $2$, $7$, $2-1=1$ nor $7-1=6$ is divisible by $5$. Hence for any prime $p\neq 2,7$ or $5$, the eigenvalues of $\Frob_p^{-1}$ on $H^3_{\mathrm{et}}(X_{-\frac{1}{7},\Qbar}, \Q_5)$ would include $1$ or $p^3$ (in fact one can easily show it would be both), contradicting the theorem of Deligne \cite{De} (formerly part of the Weil conjectures) that they must have absolute value $p^{3/2}$.

If we believe Proposition \ref{notirred}, $\rho_5$ is not irreducible. The only possibility remaining is that it has two irreducible composition factors, $\rho_1$ such that $\rhobar_1\simeq \rhobar_{g,5}(-1)$ and $\rho_2$ such that ${\rhobar_2}^{\mathrm{ss}}\simeq \mathrm{id}\oplus\epsilon^{-3}$. We need to show that $\rho_1(1)\simeq\rho_{g,5}$ and $\rho_2\simeq\rho_{f,5}$.

First we must show that they are modular. Both $\rho_1(1)$ and $\rho_2$ are odd, unramified outside finitely many places, and their restrictions to $G_{\Q_5}$ are crystalline (hence potentially semi-stable), with distinct Hodge-Tate weights. Since also $5>3$, all the hypotheses of \cite[Theorem 1.0.4]{Pa} are satisfied, with the result that $\rho_1(1)$ and $\rho_2$ are indeed modular. This very general theorem on the Fontaine-Mazur conjecture for $\GL_2$ combines new results of Pan with earlier work of Kisin \cite[Theorem]{Kis2} and of Skinner and Wiles \cite[Theorem 1.2]{SW}. Note that it avoids any assumptions about restrictions of $p$-adic representations to $G_{\Q_p}$ being potentially Barsotti-Tate or ordinary, and applies for both reducible and irreducible residual representations of $G_{\Q}$.

Now we know that $\rho_1(1)$ and $\rho_2$ are associated with some newforms. The levels of these forms are the Artin conductors of the representations, by a theorem of Carayol \cite{Ca}. The conductor of $\rhobar_1(1)\simeq \rhobar_{g,5}$ is $14$. The conductor of $\rho_1(1)$ is also divisible only by the primes $2$ and $7$ (cf. Lemma \ref{goodred}). Since neither $2$ nor $7$ is $-1\pmod{5}$, it follows from results of Livn\'e \cite[Propositions 2.3, 3.1]{L} that $\rho_1(1)$ must have the same conductor $14$. Since there is only one newform of weight $2$ for $\Gamma_0(14)$, we must have $\rho_1(1)\simeq \rho_{g,5}$. 

We now know that for every prime $p\notin\{2,5,7\}$, $\tr(\rho_1(\Frob_p^{-1}))=pb_p$, which matches $\alpha p$ throughout \cite[Table 2]{COES}. Hence $\tr(\rho_2(\Frob_p^{-1}))=\beta$ (not just mod $5$). Again by \cite[Propositions 2.3, 3.1]{L}, the conductor of $\rho_2$ divides $14$. The only newform of level dividing $14$ with Hecke eigenvalues matching the $\beta$ is $f$, i.e. $\mathbf{14.4.a.a}$.
\end{proof}
\begin{remar} Theorems 7.1.1 and 8.0.1 of \cite{Pa} apply more generally with $\Q$ replaced by a totally real $F$ (abelian over $\Q$ in the former case). But there is a condition that $p$ must split completely in $F$, which does not hold when $F=\Q(\sqrt{17})$ and $p=5$. This is among the obstacles to proving Conjecture \ref{Conj2}. \hfill$\triangle$
\end{remar}
\begin{remar} For values of $\phi\in\PP^1(\Q)\backslash\Sigma$ different from $\phi=-\frac{1}{7}$ (assumed to be the only rational rank-$2$ attractor point), we would expect that $\rho_{\ell}$ is irreducible, and is isomorphic to the $4$-dimensional Galois representation associated \cite{W} with some genus-$2$ Siegel modular form $F$, a cuspidal Hecke eigenform of weight $3$, not CAP or endoscopic.

Assume this at least for $\ell=5$. Then for primes of good reduction $p\neq 5$, $\tr(\rho_{5}(\Frob_p^{-1}))=\mu_F(T(p))$, the eigenvalue by which a Hecke operator $T(p)$ acts on $F$. The reduction $\rhobar_5$ would have composition factors $\mathrm{id}$, $\epsilon^{-3}$ and a $2$-dimensional representation, assumed irreducible, which would then be $\rhobar_{h,5}(-1)$, associated with some weight-$2$ normalised newform $h=\sum c_nq^n$. (Examination of the reductions mod $5$ of the Eulers factors in the tables in \cite[B1]{COKM2} strongly suggests that in fact $h$ is the weight-$2$ newform attached to the elliptic curve $\mathcal{E}_{1,1,1,\phi^{-1}}$ preceding Lemma \ref{smooth}. We see a similar phenomenon in Remarks \ref{MeVe} and \ref{others}.) We would have a congruence
\begin{equation}\label{cong}\mu_F(T(p))\equiv 1+p^3+pc_p\pmod{5},\end{equation}
for all primes $p\neq 5$ of good reduction.

For two examples of $F$ of paramodular level,  congruences of this shape were proved in \cite[\S 11]{DPRT}. For one of paramodular level $89$, the congruence is mod $5$ (like those here), while for $F$ of paramodular level $61$ the modulus is a divisor of $19$. The latter congruence was discovered experimentally by Buzzard and Golyshev. Both these paramodular forms have rational Hecke eigenvalues, so could conceivably be associated with Calabi-Yau threefolds, but we did not use this.

One might hope that for well-chosen examples of $\phi\in\PP^1(\Q)-\Sigma$ different from $\phi=-\frac{1}{7}$, it is possible to prove that $\rho_5$ comes from a particular $F$, in the same manner that Berger and Klosin \cite{BeKl} prove paramodularity of certain abelian surfaces. Thus, one would identify $h$ such that the right hand side of (\ref{cong}) is $\tr(\rhobar_5(\Frob_p^{-1}))$, and prove a congruence $(\ref{cong})$ for $F$. Then one would use Galois deformation theory, for a reducible residual representation, to show that the only possibility for $\rho_5$ (attached to $X_{\phi}$) is $\rho_{F,5}$ (attached to $F$). A difficulty with this approach would appear to be that, unlike the case of abelian surfaces, we do not know much about the local restrictions, at primes of bad reduction, of Galois representations coming from Calabi-Yau threefolds. Hence it would be problematic trying to select the appropriate local deformation conditions at such primes, or at least showing that they are satisfied by $\rho_5$.

Examination of the mod $2$ reductions of the Euler factors in the tables in \cite[B1]{COKM2} or \cite[C.3]{COS} suggests that if an Euler factor at a good prime $p$ is $R_p(p^{-s})^{-1}$ then always $R_p(T)\equiv 1+T^4\pmod{2}$, so that $\rhobar_2(\Frob_p)$ is of order $1, 2$ or $4$, according to \cite[5.1.6]{BPPTVY}. It would follow from \cite[Lemma 5.1.7]{BPPTVY} that $\rhobar_2$ is not absolutely irreducible, so that one cannot prove paramodularity using those authors' extension of the Faltings-Serre method.
\hfill$\triangle$
\end{remar}
\begin{remar}\label{644} Looking at the table in \cite[B2]{COKM2}, the parameter value $\phi=1/2$ has $\mathcal{E}_{1,1,1,\phi^{-1}}$ in the same isogeny class ($\mathbf{14.a}$) as for $\phi=-1/7$, hence the $L$-function attached to $H^3_{\mathrm{et}}(\overline{X}_{\aaaa,\Qbar}, \Q_5)$ (for $\phi=1/2$) has a factor $L(g, s-1)^4$, with $g=\mathbf{14.2.a.a}$ as in Conjecture \ref{Conj1}. Using the tables in \cite[B1]{COKM2}, of Euler factors at primes $11\leq p\leq 47$ (or \cite[C.3]{COS} for $5\leq p\leq 97$), it appears that the complementary factor, i.e. the $L$-function attached to $H^3_{\mathrm{et}}(X_{\frac{1}{2},\Qbar}, \Q_5)$, is $L(\rho_{F,5},s)=L(s, F, \spin)$, with a certain $F$ of (genus $2$, weight $3$ and) paramodular level $644=2^2\cdot 7\cdot 23$. It may be found about two-thirds of the way through the file \begin{verbatim} hecke_ev_3_0_644.dat \end{verbatim} at \cite{ALRTV}, with \vskip10pt
\begin{tabular}{|c|c|c|c|c|c|c|c|c|c|c|}
$p$ & $11$&$13$&$17$&$19$&$29$&$31$&$37$&$41$&$43$&$47$\\\hline
$\mu_F(T(p))$ &$2$&$-14$&$-4$&$-12$&$156$&$68$&$-162$&$128$&$-38$&$-440$
\end{tabular} .
\vskip10pt
The irreducibility of $\rho_{F,5}$ is known, as in \cite[Theorem 1.1(i) and \S 2.4]{Wei}. If it were possible to use this to produce a non-zero element of the Selmer group $H^1_{\{2,7,23\}}(\Q, W)$ (notation as in the ``proof'' of Proposition \ref{notirred}) then, checking also the Euler factor $L_{23}(g,-3)$, it would come from $H^1_f(\Q, W)$, contrary to what we already found, that $\ord_5\left(\frac{L(E_5,1)}{\Omega_{E_5}}\right)=0$. This example shows that any attempt to prove Proposition \ref{notirred} by showing that irreducibility of $\rho_5$ would be enough to produce a Selmer group element, is doomed to failure. It was this example that revealed the error in the proof of Proposition \ref{notirred}.
\hfill$\triangle$
\end{remar}

{\bf Acknowledgements. } I would like to thank Albrecht Klemm and Duco van Straten for very useful discussions (and the letter \cite{vS}), which had a great influence on Section 2 and Lemma 3.1. I thank also Kilian B\"onisch for referring me to \cite{BEKK} and Marc Mezzarobba for his helpful response to an enquiry. I am grateful to two anonymous referees for their careful reading and valuable comments.

{\bf Computational supplement}
\vskip10pt

{\bf Magma code for finding the intersection of the kernels of $T_{11}$, $T_{13}-1$, $T_{17}-1$, $T_{19}-2$ inside $\mathbb{S}_2(\Gamma_0(2^87^2))$, in the proof of Proposition 5.3.}
\vskip10pt
\begin{verbatim}
M:=ModularSymbols(12544, 2, GF(5));
C:=CuspidalSubspace(M);
P<x>:=PolynomialRing(GF(5));
I:=[<11,x>, <13, x-1>, <17, x-1>, <19, x-2>];
K:=Kernel(I, C);
K;
\end{verbatim}
\vskip10pt

 {\bf Magma code for finding the exceptional divisor above $P$ in the blowup of $P$ in $X_{\mathbf{a}}$, in the proof of Lemma 4.2.}
\vskip10pt
\begin{verbatim}
k:=Rationals();
A<x,y,z,w>:=AffineSpace(k,4);
V:=Scheme(A,(1+x+x*y+x*y*z+x*y*z*w)*(1+w+w*z+w*z*y+w*z*y*x)+7*x*y*z*w);
p:=A ! [-1,0,0,-1];
X,mp:=Blowup(V,p);
Y:=p @@ mp;
MinimalBasis(Ideal(Y));
\end{verbatim}
\vskip10pt

{\bf Magma code to compute $\#\overline{X}_{\mathbf{a}}(\FF_{p^2})$ with $p=113$, for the proof of Proposition 4.5.}  We have exploited the symmetry between $a,b$ and $c$ to cut the computation time by a factor of $6$, and have divided it into subsums to facilitate interrupting and resuming the computation. 
At the end $U=\#(\overline{X}_{\mathbf{a}}\cap T)(\FF_{113^2})$, $V=\#(\overline{X}_{\mathbf{a}})(\FF_{113^2})$, and the number on the last line is that to be compared with the number in the fifth column of the table. 
\vskip10pt
\begin{verbatim}
F<x>, mp:=ExtensionField<GF(113),x|x^2+x+1>;
increment:=function(a,b,c);
d:=((1+a+b+c)*((1/a)+(1/b)+(1/c)+1)+6)^2+28;
if IsSquare(d) then;
if d eq Zero(F) then;
if (a eq b) and (b eq c) then;
return 1;
else if (a eq b) or (a eq c) or (b eq c) then;
return 3;
else return 6;
end if ;
end if;
else
if (a eq b) and (b eq c) then;
return 2;
else if (a eq b) or (a eq c) or (b eq c) then;
return 6;
else return 12;
end if ;
end if;
end if;
else return 0;
end if;
end function;
U:=0;
for e in [0..112] do
S:=0;
for f in [0..112] do;
A:=e+f*x;
for g in [e..112] do;
for h in [0..112] do;
B:=g+h*x;
for m in [g..112] do;
for n in [0..112] do;
C:=m+n*x;
if (A ne Zero(F)) and (B ne Zero(F)) and (C ne Zero(F)) then;
if ((g ne e) or (f le h)) and ((g ne m) or (h le n)) then;
S:=S+increment(A,B,C);
end if;
end if;
end for;
end for;
end for;
end for;
end for;
U:=U+S;
e, S, U;
end for;
V:=U+48*p^4+46*p^2+14;
V;
(1+p^6+45*(p^4+p^2))-V;
\end{verbatim}

\end{document}